\pdfoutput=1
\documentclass[11pt]{amsart}

\usepackage[english]{babel}
\usepackage{amsmath, amssymb, fullpage}

\usepackage[activate={true,nocompatibility},final,tracking=true,kerning=true,spacing=true,factor=1100,stretch=10,shrink=10]{microtype}

\usepackage{amsthm}
\usepackage{thmtools}
\usepackage{mathtools}
\usepackage{stmaryrd}

\usepackage{xcolor}
\usepackage{graphicx}
\usepackage{tikz}
\usetikzlibrary{arrows.meta}
\usepackage{tkz-euclide}

\usepackage[boxruled, longend]{algorithm2e}

\usepackage{enumitem}
\usepackage[justification=centering]{caption}
\usepackage{subcaption}
\usepackage{afterpage}

\usepackage[ocgcolorlinks]{hyperref}
\usepackage{cleveref}
\usepackage{upref}
\crefdefaultlabelformat{#2\textup{#1}#3}

\newtheorem{theorem}{Theorem}[section]

\newtheorem{proposition}[theorem]{Proposition}
\newtheorem{corollary}[theorem]{Corollary}
\newtheorem{lemma}[theorem]{Lemma}

\theoremstyle{definition}
\newtheorem{definition}[theorem]{Definition}

\newtheorem{remark}[theorem]{Remark}

\newcommand{\N}{\mathbb{N}}
\newcommand{\Z}{\mathbb{Z}}
\newcommand{\Q}{\mathbb{Q}}
\newcommand{\R}{\mathbb{R}}
\newcommand{\C}{\mathbb{C}}

\newcommand{\Hyp}{\mathbb{H}^2}
\newcommand{\cH}{\mathcal{H}}

\newcommand{\e}{\varepsilon}

\newcommand{\btw}{\,{-}\,}
\DeclarePairedDelimiter\cyc{\llbracket}{\rrbracket}

\renewcommand{\setminus}{\smallsetminus}

\newcommand{\Magma}{\textsc{Magma}}

\DeclareMathOperator{\SL}{SL}
\DeclareMathOperator{\PSL}{PSL}

\DeclareMathOperator{\lmost}{\mathcal{L}}
\DeclareMathOperator{\rmost}{\mathcal{R}}

\DeclareMathOperator{\Cay}{Cay}

\DeclareMathOperator{\csch}{csch}
\DeclareMathOperator{\sech}{sech}

\DeclarePairedDelimiter{\abs}{\lvert}{\rvert}

\DeclarePairedDelimiter{\gen}{\langle}{\rangle}

\DeclarePairedDelimiter{\stdrightbar}{}{\rvert}
\DeclarePairedDelimiter{\fullrightbar}{.}{\rvert}
\makeatletter
\newcommand\rightbar{\@ifstar{\fullrightbar*}{\stdrightbar}}
\makeatother

\newcommand\SetSymbol[1][]{{\nonscript\;}{:}{\nonscript\;}\mathopen{}}
\providecommand\given{} 
\DeclarePairedDelimiterX{\Set}[1]\{\}{\renewcommand\given{\SetSymbol[\delimsize]}#1}

\microtypecontext{spacing=nonfrench}
\SetTracking{encoding={*}, shape=sc}{0}

\setlist[enumerate,1]{label=\textup{(\arabic*)}}

\colorlet{link}{blue}
\definecolor{cite}{HTML}{009900}
\hypersetup{
  colorlinks=true,
  linkcolor=link,
  citecolor=cite,
}
\tikzset{vertex/.style={circle,fill,inner sep=0pt,minimum size=1.7mm}, edge/.style={semithick}, >=stealth}

\DontPrintSemicolon
\SetKwBlock{Loop}{loop}{end}
\SetKw{Continue}{continue}

\definecolor{NiceGreen}{HTML}{00CC00}
\definecolor{NiceRed}{HTML}{EE0000}
\definecolor{NiceBlue}{HTML}{0033FF}

\begin{document}

\title{Recognition and constructive membership for discrete subgroups of $\SL_2(\R)$}
\author{Ari Markowitz}
\address{
Department of Mathematics, University of Auckland, Auckland, New Zealand}
\email{ari.markowitz@auckland.ac.nz}
\begin{abstract}
We provide algorithms to decide whether a finitely generated subgroup of $\SL_2(\R)$ is discrete, solve the constructive membership problem for finitely generated discrete subgroups of $\SL_2(\R)$, and compute a fundamental domain for a finitely generated Fuchsian group. These algorithms have been implemented in \textsc{Magma} for groups defined over real algebraic number fields.
\end{abstract}

\maketitle

\section{Introduction}
The problem of deciding discreteness of a finitely generated linear group has been explored in various settings. The generating pairs of 2-generated Fuchsian groups were characterised in a series of papers by Rosenberger, Kern-Isberner, Purzitsky and Matelski \cite{rosenberger-80, rosenberger-83, matelski, purzitsky-72, purzitsky-72-2, purzitsky-74, purzitsky-76, purzitsky-rosenberger, purzitsky-rosenberger-73, rosenberger-86}. This characterisation takes the form of an algorithm, in which a core procedure is the repeated application of Nielsen transformations to the generators that minimise their traces. Using this approach, Gilman and Maskit presented an explicit algorithm that decides whether a 2-generated subgroup $\SL_2(\R)$ is discrete and free \cite{gilman-maskit, gilman}. Similarly, Kirschmer and Rüther gave an explicit algorithm that decides whether a 2-generated subgroup of $\SL_2(\R)$ is discrete \cite{kirschmer}. Similar results have been developed in the non-Archimedean setting \cite{conder, conder-schillewaert, markowitz}. All of these algorithms make use of a \emph{replacement method}, which iteratively replaces generators in the generating set until some condition is satisfied.

An alternative to the replacement method is Riley's algorithm to decide whether a finitely-generated subgroup of $\PSL_2(\R)$ is discrete \cite{riley}. Riley's algorithm applies a ``brute force'' search for generators of a \emph{Dirichlet domain} \cite[\textsection 9.4]{beardon}, halting when it finds either such a generating set or a pair of words that violates Jørgensen's Inequality \cite{jorgensen}. While this algorithm is more general, the number of words to be considered is exponential in the maximum length of a word the algorithm must consider. By contrast, the algorithm of Gilman and Maskit has complexity which is polynomial in the maximal length of a word in the final generating set \cite{jiang}.

Closely related to this decision problem is the \emph{constructive membership problem}: given a finitely generated subgroup $G$ of a group $P$ and some $g \in P$, decide whether or not $g \in G$, and if so then write $g$ as word in a specified generating set of $G$. The constructive membership problem was solved for 2-generator discrete free subgroups of $\SL_2(\R)$ by Eick, Kirschmer, and Leedham-Green \cite{eick}, for 2-generator discrete subgroups of $\SL_2(\R)$ by Kirschmer and Rüther \cite{kirschmer}, and for finitely generated discrete free subgroups of the automorphism group of a locally finite tree by the author \cite{markowitz}. In \cite{eick} the authors pose the open problem of extending their methods to groups of arbitrary rank.

In this paper we provide algorithms to decide discreteness of finitely generated discrete subgroups of $\SL_2(\R)$, and solve the constructive membership problem for such groups. The first algorithm uses a replacement method to decide whether a group is discrete and torsion-free, which has at most linear complexity with respect to the maximum length of a word that it must consider.

\begin{restatable}{theoremx}{recognition}\label{recognition}
  Let $X$ be a finite subset of $\SL_2(\R)$ or $\PSL_2(\R)$. There exists an algorithm that decides whether or not $X$ generates a discrete group.
\end{restatable}
\begin{restatable}{theoremx}{membership}\label{membership}
  Let $P$ be $\SL_2(\R)$ or $\PSL_2(\R)$. Let $X$ be a finite subset of $P$ generating a discrete group, and let $g \in P$. There exists an algorithm that decides whether $g$ can be written as a word in $X$, and if so returns such a word.
\end{restatable}

Our approach is as follows. We first solve the recognition and constructive membership problems for finitely generated discrete torsion-free subgroups of $\PSL_2(\R)$ (which straightforwardly extends to $\SL_2(\R)$, as shown by \cite[Lemma 4.1]{eick}). Our solution is based on the approach for groups acting on trees in \cite{markowitz, weidmann}, which in turns generalises Nielsen's method for finding a free basis for a subgroup of a free group \cite[\textsection 2]{lyndon-schupp}. For trees, this method works by identifying the generators with paths on the tree, and iteratively applying Nielsen transformations to minimise the paths with respect to a given pre-well-ordering.

We apply this approach to the action of $\PSL_2(\R)$ on the hyperbolic plane $\Hyp$. Given a finite generating set $X$ of a subgroup $G$ of $\PSL_2(\R)$, we consider a map from the Cayley graph of the free group of reduced words in $X$ to $\Hyp$. \Cref{alg:free} iteratively applies one of a finite set of Nielsen transformations to $X$ to minimise the edge lengths of the corresponding Cayley graph. After a finite number of iterations, the algorithm finds either an elliptic element, or an indiscrete 2-generated subgroup, or a generating set for $G$ on which no more such transformations can be performed. In the third case, the group is discrete and torsion-free.

We show that the generating set obtained by this algorithm can be used to compute a fundamental domain, giving an algorithm similar to \cite[Algorithm 1]{eick} that solves the constructive membership problem for finitely generated discrete torsion-free subgroups of $\PSL_2(\R)$.

To decide discreteness of $G$, we apply Selberg's Lemma to find a finite-index torsion-free subgroup $H$ of $G$. Deciding discreteness of $H$ is then equivalent to deciding discreteness of $G$. By applying the constructive membership algorithm to the cosets of $H$ in $G$, we solve the constructive membership problem for $G$.

Our algorithm returns a ``signature'' that certifies that the group is either discrete (or discrete and torsion-free) or not. In the positive case, this is a \emph{reduced} generating set for a finite-index discrete torsion-free subgroup of $G$ together with a set of coset representatives; in the negative case, this is a 1- or 2-generated indiscrete subgroup of $G$. If $G$ is discrete, then we provide algorithms to compute a convex fundamental domain for $G$.

If the input matrices have real number entries, then we assume that the algorithms are executed by an oracle capable of exact comparisons of real numbers. For matrices over an algebraic number field, the algorithms can be implemented with exact precision; see Section \ref{sec:exact}.

Voight described an algorithm to compute a fundamental domain for Fuchsian groups \cite{voight} and implemented it for arithmetic groups in \Magma\ \cite{magma}.
We have implemented the remaining algorithms in \Magma\ for groups defined over real algebraic number fields \cite{markowitz-magma}. In \Cref{sec:complexity} we give a brief complexity analysis of the algorithms. In \Cref{sec:magma} we report the performance of the implementations.

\section{Preliminaries}
\subsection{Notation}
We fix the following terminology:
\begin{itemize}
  \item $\Hyp$ denotes the hyperbolic plane.
  \item $X$ is a finite subset of $\PSL_2(\R)$ generating a group $G$.
  \item $X^- = \Set{x^{-1} \given x \in X}$ and $X^\pm = X \cup X^-$.
  \item $v$ is a (distinguished) vertex of $\Hyp$. It is chosen arbitrarily, but once chosen it remains fixed.
  \item $d \colon \Hyp \times \Hyp \to [0, \infty)$ is the hyperbolic distance metric.
  \item If $g \in \PSL_2(\R)$, then the \emph{displacement} of $g$ is $\abs{g} = d(v, gv)$. Note that this differs from the \emph{translation length} of $g$, which is the minimum of $d(w, gw)$ over all $w \in \Hyp$ (see \cite[p.\ 112]{beardon}).
  \item If $x, y \in \Hyp$, then $[x, y]$ denotes the unique geodesic from $x$ to $y$.
  \item $F_X$ is the free group of reduced words in $X$.
  \item $\Cay(F_X)$ and $\Cay(G, X)$ are the geometric realisations of the Cayley graphs of $F_X$ and $G$ respectively with generating set $X$.
  \item Define a map $\phi \colon \Cay(F_X) \to \Hyp$ that sends each vertex $g$ to $gv$ and restricts on each edge to a scaled isometry. This implies that the image of an edge of $\Cay(F_X)$ is a geodesic in $\Hyp$. Note that $\phi$ factors through $\Cay(G, X)$, induced by the surjection $F_X \twoheadrightarrow G$.
  \item Let $\mu$ be the metric on $\Cay(F_X)$ inherited from $\Hyp$: namely, $\mu(a, b)$ is the arc length of $\phi \circ \gamma$, where $\gamma$ is the geodesic from $a$ to $b$. Note that $\mu(a, b) = d(\phi(a), \phi(b))$ whenever $a$ and $b$ lie on a common edge.
  \item $T$ is the metric space $(\Cay(F_X), \mu)$.
  \item Suppose $S \subseteq T$ and $A \subseteq \Hyp$. By abuse of notation, define $S \cap A = \Set{s \in S \given \phi(s) \in A}$.
  \item Given $a, b \in T$, the \emph{segment} $[a, b]_T$ is the geodesic from $a$ to $b$ in $T$.
\end{itemize}

The group generated by
\[
  A =
  \begin{bmatrix}
  1 & 2 \\ 0 & 1
  \end{bmatrix}, \qquad B =
  \begin{bmatrix}
  1 & 0 \\ 2 & 1
  \end{bmatrix}
\]
is used for illustrative purposes in Figures \ref{fig:cayley-graph}, \ref{fig:leftmost-path}, \ref{fig:bad-cayley-graph}, and \ref{fig:fundamental-domain}.
\Cref{fig:cayley-graph} shows the image of $\phi$ as represented in the Poincaré disk for $v=0$ and $X = \{A, B\}$.

\begin{figure}
  \centering
  \begin{tikzpicture}
    \node {\includegraphics[width=6cm]{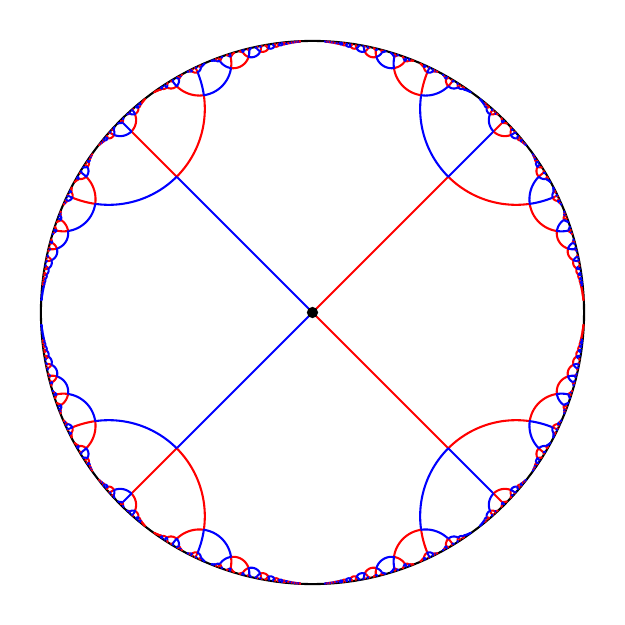}};
    \node[vertex, label={right:$v$}] at (0,0) {};
    \node[vertex, label={left:$Av$}] at (1.3,1.3) {};
    \node[vertex, label={left:$A^{-1}v$}] at (1.3,-1.3) {};
    \node[vertex, label={right:$B^{-1}v$}] at (-1.3,1.3) {};
    \node[vertex, label={right:$Bv$}] at (-1.3,-1.3) {};
  \end{tikzpicture}
  \caption{The embedding of a Cayley graph in $\Hyp$}
  \label{fig:cayley-graph}
\end{figure}

Throughout this paper we freely use the Jordan curve theorem \cite{maehara}.

\section{Reduced generating sets}

In this section we present the theory needed to describe and prove the validity of the recognition algorithm.

\subsection{Bounding paths}
There is a natural notion of angle defined for triples of points in $\Hyp$, inherited from angles in Euclidean space \cite[\textsection 7.6]{beardon}. Furthermore, angles in $\Hyp$ also inherit an \emph{orientation} from the Euclidean plane. As in Euclidean space, all angles of a simple polygon in $\Hyp$ (when the vertices are ordered in a connected cycle) have the same orientation.

Given $p, q, r \in \Hyp$, define $\measuredangle_p (q, r)$ to be the clockwise angle at $p$ from $[p, q]$ to $[p, r]$. Note that the ordering of $(q, r)$ matters: if $\measuredangle_p (q, r) \neq 0$, then $\measuredangle_p (q, r) = 2\pi - \measuredangle_p (r, q)$. We use this operation to define a cyclic order (a ternary relation that arranges elements ``on a circle''; see \cite{huntington}) on elements of $\Hyp$.

\begin{definition}
Let $B \subset X^\pm \times X^\pm \times X^\pm$ be a cyclic order such that $B(x, y, z)$ holds if
\begin{enumerate}
  \item $\measuredangle_v (xv, yv) < \measuredangle_v (xv, zv)$;
  \item $\measuredangle_v (xv, yv) = \measuredangle_v (xv, zv)$, $yv \neq zv$, and $xv \notin [yv, zv]$;
\end{enumerate}
\end{definition}

Note that (2) implies that $\{v, yv, zv\}$ are collinear. To ensure that $B$ is a cyclic order, for each $y \in X^\pm$ we must arbitrarily choose a linear order on $\Set{z \in X^\pm \given yv = zv}$. The specific choice made does not affect the arguments in this paper; all instances in which $\{v, yv, zv\}$ are collinear are handled by \Cref{lem:collinear}. We fix this cyclic order for the rest of the paper.

\begin{definition}
  Let $P = (g_1, \dots, g_n)$ be a path on $T$. For $1 \leq i \leq n$, let $x_i \in X^\pm$ be such that $g_{i+1} = x_i g_i$. If $x_{i+1}$ directly follows $x_i^{-1}$ in the cyclic order $B$ for all $i$ (that is, there is no $y \in X^\pm$ such that $B(x_i^{-1}, y, x_{i+1})$ holds), then $P$ is a \emph{leftmost path}.
\end{definition}

We give a geometric interpretation, as shown in \Cref{fig:leftmost-path}. For every subpath of the form $(g_1, g_2, g_3)$ of a leftmost path, either $\{g_1v, g_2v, g_3v\}$ is collinear or $g_3$ is the neighbour of $g_2$ minimising the nonzero angle $\measuredangle_{g_2v}(g_1v, g_3v)$. Essentially, a leftmost path ``turns left'' at each vertex by the minimum interior angle.

\begin{figure}
  \centering
  \begin{tikzpicture}
    \node {\includegraphics[width=8cm]{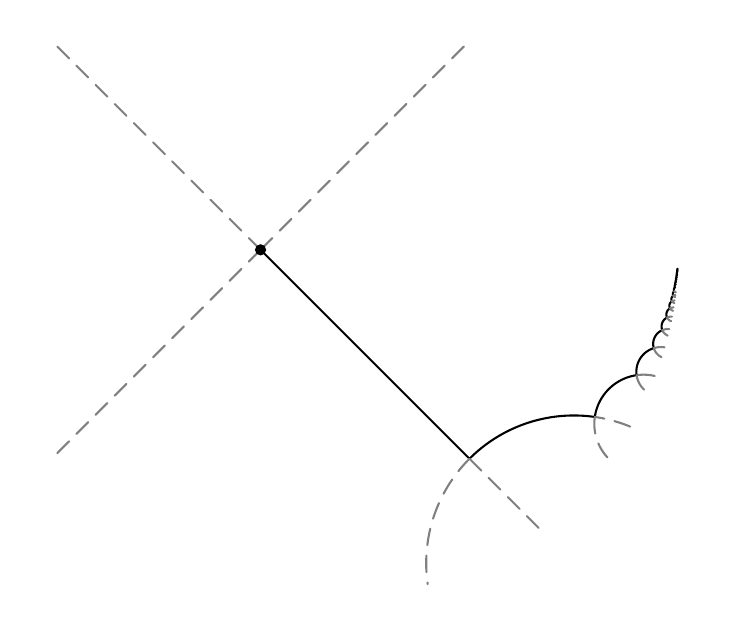}};
    \node[vertex, label={right:$g_1 v$}] at (-1.14,0.7) {};
    \node[vertex, label={[xshift=2]above:$g_2 v$}] at (1.15, -1.6) {};
    \node[vertex, label={[xshift=-5]above:$g_3 v$}] at (2.53, -1.15) {};
    \node[rotate=55] at (3, -0.3) {$\cdots$};
    \node at (-0.2, -0.65) {$e$};
  \end{tikzpicture}
  \caption{An infinite leftmost path $\lmost(e)$ in \Cref{fig:cayley-graph}}
  \label{fig:leftmost-path}
\end{figure}

\begin{definition}
  The reverse of a leftmost path is a \emph{rightmost path}. A \emph{bounding path} of $T$ is a leftmost or rightmost path. An infinite path is a \emph{bounding path} if every finite subpath is a bounding path.
\end{definition}

Given a directed edge $e$ of $T$, define $\lmost(e)$ to be the infinite leftmost path with first edge $e$. Similarly, define $\rmost(e)$ to be the infinite rightmost path with first edge $e$.

If $e = (g, xg)$ for some $g \in F_X$ and $x \in X^\pm$, then $x$ is the \emph{type} of $e$. Note that edge type is invariant under the action of $F_X$.

\begin{proposition}\label{prop:leftmost-repeats}
  Let $P = (g_i)_{i \in \N}$ be an infinite bounding path on $T$. The sequence $(x_i)_{i \in \N}$ of edge types of $P$ is cyclic: there exists $n > 0$ such that $x_{i+n} = x_i$ for all $i \in \N$. If $n$ is minimal, then $x_i \neq x_j$ for all $i < j < i+n$.
\end{proposition}

\begin{proof}
  We assume $P$ is a leftmost path; the proof for rightmost paths is similar.
  Define $\eta \colon X^\pm \to X^\pm$ sending $x$ to the element directly following $x^{-1}$ in the cyclic ordering of $X^\pm$. Observe that $\eta$ is a permutation of $X^\pm$: it is the composition of the involution $x \mapsto x^{-1}$ and a cyclic permutation. Since $x_{i+1} = \eta(x_i) = \eta^{i-1}(x_1)$ for all $i \in \N$, and $X^\pm$ is finite, the proposition follows. Note that the length of the repeated sequence is the order of $x_1$ under $\eta$.
\end{proof}

\begin{definition}
  Let $P = (x_1, x_2, \dots)$ be an infinite bounding path on $T$. Let $n$ be as in \Cref{prop:leftmost-repeats}. The word $x_1 \dots x_n$ is the \emph{principal word} of $P$.
\end{definition}

Let $P$ be an infinite bounding path, and let $P'$ be a subpath of $P$ with the same initial vertex. The \emph{principal word} of $P$ is the principal word of $P'$. If $Q$ is a bi-infinite bounding path, then each principal word of a subpath of $Q$ is a principal word of $Q$. Given two principal words of $Q$, one can be transformed into the other by cyclically permuting its terms.

Suppose $P$ is a bi-infinite path with principal word $g = x_1 \dots x_n$. Let $h$ be a word in $X$ such that $(h, x_1h)$ is an edge of $P$. We see that $h^{-1}P$ contains the vertex $g^k$ for all $k \in \Z$; hence every bi-infinite path on $T$ is isometric in $F_X$ to a path stabilised by its principal word.

\begin{definition}
  A nontrivial subword of a principal word is \emph{short}.
\end{definition}

\subsection{Self-intersecting paths}

A difficulty in characterising $G$ using $X$ is that the map from $\Cay(G, X)$ to $\Hyp$ need not be injective, as shown in \Cref{fig:bad-cayley-graph}; this can occur regardless of whether or not $G$ is discrete. We seek a generating set $X'$ such that $\Cay(G, X')$ embeds into $\Hyp$. In other words, we attempt to eliminate the following obstruction:

\begin{figure}
\centering
\begin{tikzpicture}
  \node {\includegraphics[width=6cm]{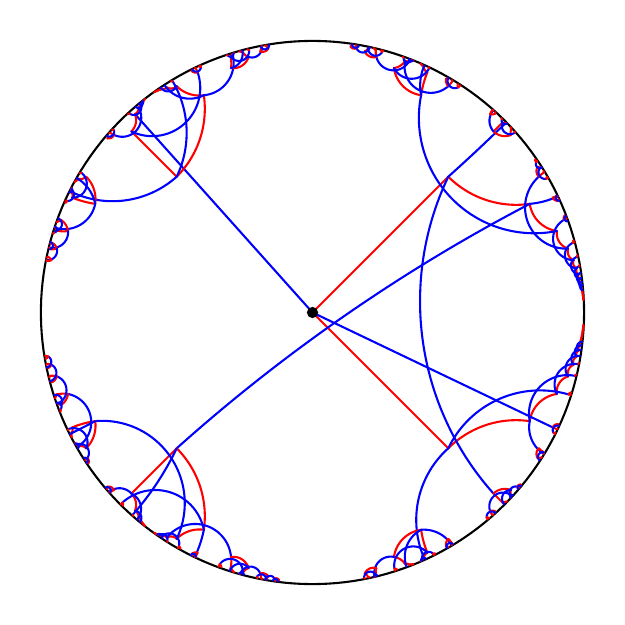}};
  \node[vertex, label={below left:$v$}] at (0,0) {};
\end{tikzpicture}
\caption{The embedding of a worse-chosen Cayley graph of the group in \Cref{fig:cayley-graph}}
\label{fig:bad-cayley-graph}
\end{figure}

\begin{definition}\label{def:arc-map}
  A \emph{self-intersection} of $T$ is a point $w \in \phi(T)$ such that $\abs{\phi^{-1}(w)} > 1$.
\end{definition}

The results in this section give conditions that determine the existence of self-intersections. While self-intersections are not detected directly in the algorithms in this paper, they are used to prove that the algorithms give the desired output.

Recall that $\PSL_2(\R)$ is isomorphic to the group of orientation-preserving isometries of $\Hyp$. Also recall from \cite[Definition 4.3.2]{beardon} that every nontrivial $g \in \PSL_2(\R)$ is:

\begin{itemize}
  \item \emph{elliptic} if $g$ fixes no points on the boundary $\partial\Hyp$ (and hence fixes one point on $\Hyp$);
  \item \emph{parabolic} if $g$ fixes one point on $\partial\Hyp$; or
  \item \emph{hyperbolic} if $g$ fixes two points on $\partial\Hyp$.
\end{itemize}

A \emph{level} curve of $g$ is a curve of constant curvature stabilised by $g$. The level curves of $g$ partition $\Hyp$. If $g$ is elliptic, then its level curves are circles centred at the fixed point of $g$; if $g$ is parabolic, then its level curves are \emph{horocycles} tangent to the fixed boundary point of $g$; and if $g$ if hyperbolic, then its level curves are \emph{hypercycles} parallel to the translation axis of $g$ \cite[\textsection 7.27]{beardon}.

The following results demonstrate \emph{local-to-global} properties of $\Hyp$ and finite generating sets of $\PSL_2(\R)$. Our goal is to show that if $T$ has a self-intersection, then one of a finite set of paths on $T$ has a self-intersection. These results are analogous to the local-to-global properties of \cite[Lemma 2.12]{markowitz}.

\begin{lemma}\label{lem:collinear}
  If there exists distinct $x, y \in X^\pm$ such that $x \neq y^{-1}$ and $\{v, xv, yv\}$ are collinear, then there exists a self-intersecting bounding path on $T$.
\end{lemma}

\begin{proof}
  Let $y' \in X^\pm$ directly follow $x$ with respect to the cyclic ordering $\cyc{\cdot, \cdot, \cdot}$. Either $y' = y$ or $\cyc{x, y', y}$; in either case $\{v, xv, y'v\}$ is collinear. The path $(v, x^{-1}v, y'v)$ is therefore a self-intersecting bounding path on $T$.
\end{proof}

Bounding paths on $T$ provide a ``bound'' on the long-term behaviour of paths on $T$, in the following sense:

\begin{proposition}\label{prop:path-impl-clock}
  If $T$ has a self-intersection, then there exists a self-intersecting bounding path on $T$.
\end{proposition}

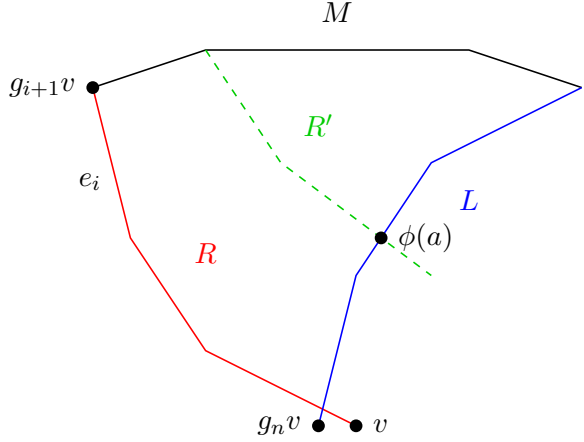
\begin{figure}[t]
  \centering
  \begin{tikzpicture}
    \coordinate (r1) at (0, 0);
    \coordinate (r2) at (-2, 1);
    \coordinate (r3) at (-3, 2.5);
    \coordinate (r4) at (-3.5, 4.5);
    \coordinate (m1) at (-2, 5);
    \coordinate (m3) at (1.5, 5);
    \coordinate (l1) at (3, 4.5);
    \coordinate (l2) at (1, 3.5);
    \coordinate (l3) at (-0, 2);
    \coordinate (l4) at (-0.5, 0);
    
    \coordinate (s2) at (-1, 3.5);
    \coordinate (s3) at (1, 2);
    
    \draw[edge, red] (r1) -- (r2) -- (r3) --node[below left, black]{$e_i$} (r4);
    \draw[edge] (r4) -- (m1) -- (m3) -- (l1);
    \draw[edge, NiceGreen, dashed] (m1) -- (s2) -- (s3);
    \draw[edge, blue] (l1) -- (l2) -- (l3) -- (l4);
    \node[vertex, label=right:$v$] at (r1) {};
    \node[vertex, label=left:$g_n v$] at (l4) {};
    \node[vertex, label=left:$g_{i+1} v$] at (r4) {};
    \tkzInterLL(l2,l3)(s2,s3) \tkzGetPoint{a}
    \node[vertex, label=right:$\phi(a)$] at (a) {};
    \node[red] at (-2, 2.3) {$R$};
    \node[blue] at (1.5, 3) {$L$};
    \node at (-0.25, 5.5) {$M$};
    \node[NiceGreen] at (-0.5, 4) {$R'$};
  \end{tikzpicture}
  \caption{The construction of $P$ for \Cref{prop:path-impl-clock}}
  \label{fig:intersection}
\end{figure}

\begin{proof}
  Suppose for a contradiction that there is no self-intersecting bounding path on $T$.
  Let $w$ be a self-intersection point of $T$. Let $P$ be a minimal self-intersecting path (that is, a self-intersecting path such that no subpath is self-intersecting) with intersection point $w$. Such a $P$ must exist, since every path has finitely many subpaths. Then $P$ bounds a simply-connected subset $D$ of $\Hyp$.
  Let $(e_1, \dots, e_{n})$ be the sequence of edges in $P$.
  If some interior angle of $D$ is 0, then three consecutive vertices of $P$ are collinear so, by \Cref{lem:collinear}, there exists a self-intersecting bounding path on $T$. Hence we may assume that $D$ is a simple polygon: a closed piecewise-linear curve with no self-intersections. Depending on the orientation of the angles in $P$, either $\lmost(e_i)$ has an initial segment intersecting the interior of $D$ and not the exterior for every $1 \leq i < n$, or $\rmost(e_i)$ has this same property for all such $i$. We assume the former, as the argument in either case is similar.
  
  We use the following construction, as shown in \Cref{fig:intersection}:
  Let $n(P)$ be the smallest integer such that $P$ is the concatenation of paths $P = R \oplus M \oplus L$ where $R$ is a rightmost path, $L$ is a leftmost path, and $M$ is a path of length $n(P)$. Note that $n(P) \leq \abs{P} - 2$ where $\abs{P}$ is the number of edges in $P$, since a single edge is both a leftmost and rightmost path. Let $e_i$ denote the last edge of $R$. Let $L'$ be the minimal path containing the longest segment $[g_{i+1}, a]_T$ of $\lmost(e_i)$ whose image under $\phi$ is a subset of $D$. It follows that $\phi(a) \in D$. If $a$ intersects $R$, then $P' = R \oplus L'$ satisfies $n(P') = 0$. Otherwise, take $R'$ as the initial segment of the reverse of $L'$ such that the final vertex of $R'$ is in $P$. If $R'$ intersects $L'$, then we write $P' = R' \oplus M \oplus L''$, where $L''$ is an initial segment of $L$. If $R'$ intersects $M'$, then we write $P' = R' \oplus M' \oplus L''$, where $L''$ is a leftmost path and $M' \oplus L'' = M$. In either case, $n(P') \leq n(P)$, and $P'$ is a self-intersecting path bounding a region $D' \subset D$ of smaller area than $D$.
  
  Repeating this process, we obtain a sequence $P_1, P_2, \dots$ of minimal self-intersecting paths, bounding regions of decreasing area such that $n(P_{i+1}) \leq n(P_i)$ for all $i \in \N$. However there can only be finitely many minimal self-intersecting paths $P_i$ such that $n(P_i) \leq n(P)$, since $P_i = R_i \oplus M_i \oplus L_i$ is determined up to isometry by the edge types of $M_i$, together with the two edges adjacent to $M_i$. By contradiction, $T$ has a self-intersecting bounding path.
\end{proof}

\begin{definition}
  Let $[a, b]_T$ be a segment of a bounding path with principal word $g$. If $\mu(a, b) \leq [1, g]_T$, then $[a, b]_T$ is \emph{short}.
\end{definition}

Notice that if $g \in G$ is short and $P = [1, g]_T$ self-intersects, then $P$ is short. If $P$ is the minimal path containing some short segment, then the sequence of edge types of $P$ is either of the form $(x_1, \dots, x_i)$ or $(x_1, \dots, x_n, x_1)$, where the $x_i$ are distinct.

\begin{lemma}\label{lem:intersection0}
  Let $g \in \PSL_2(\R)$ be hyperbolic or parabolic. Let $\gamma \colon [0, 1] \to \Hyp$ be a simple curve with image $C$. If $C$ intersects $g^kC$ for some $k \in \N$, then $C$ intersects $gC$.
\end{lemma}

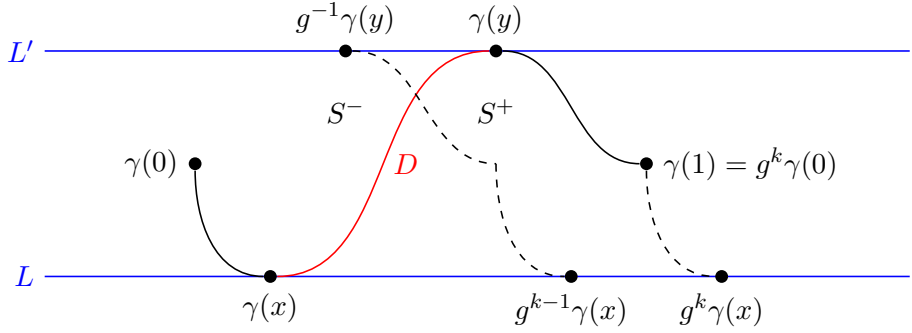
\begin{figure}[t]
  \centering
  \begin{tikzpicture}
    \draw[edge, blue] (-6, -1.5) -- (5.5, -1.5) node[pos=0, left] {$L$};
    \draw[edge, blue] (-6, 1.5) -- (5.5, 1.5) node[pos=0, left] {$L'$};
    
    \node[vertex, label={left:$\gamma(0)$}] (a) at (-4, 0) {};
    \node[vertex, label={below:$\gamma(x)$}] (x) at (-3, -1.5) {};
    \node[vertex, label={above:$\gamma(y)$}] (y) at (0, 1.5) {};
    \node[vertex, label={right:$\gamma(1) = g^k\gamma(0)$}] (b) at (2, 0) {};
    
    \node[vertex, label={below:$g^k\gamma(x)$}] (xk) at (3, -1.5) {};
    
    \node[vertex, label={above:$g^{-1}\gamma(y)$}] (y-) at (-2, 1.5) {};
    \coordinate (b-) at (0, 0) {};
    \node[vertex, label={below:$g^{k-1}\gamma(x)$}] (xk-) at (1, -1.5) {};

    \draw[edge] (a) to[out=270, in=180] (x);
    \draw[edge, red] (x) to[out=0, in=180] node[right] {$D$} (y);
    \draw[edge] (y) to[out=0, in=180] (b);
    \draw[edge, dashed] (b) to[out=270, in=180] (xk);
    
    \draw[edge, dashed] (y-) to[out=0, in=180] (b-);
    \draw[edge, dashed] (b-) to[out=270, in=180] (xk-);
    
    \node at (0, 0.7) {$S^+$};
    \node at (-2, 0.7) {$S^-$};
    
  \end{tikzpicture}
  \caption{The construction in \Cref{lem:intersection0}}
  \label{fig:short-intersection}
\end{figure}

\begin{proof}
  We may assume $k$ is minimal. For contradiction, suppose $k > 1$. Also assume (by restricting and reparameterising $\gamma$ if necessary) that the images of $\gamma$ and $g^k\gamma$ intersect only at $\gamma(1) = g^k\gamma(0)$.
  
    We construct the diagram shown in \Cref{fig:short-intersection}. Let $L$ and $L'$ be level sets of $g$ intersecting $C$. We may assume that $L$ and $L'$ are chosen to maximise $d(L, L')$. Let $x, y \in [0, 1]$ be such that $D = \gamma([x, y])$ intersects $L$ and $L'$ only at $\gamma(x)$ and $\gamma(y)$ respectively. Assume (swapping $L$ and $L'$ if necessary) that $x \leq y$.
  
  Let $S$ be the closed strip bounded by $L$ and $L'$. Let $S^-$ and $S^+$ be the closure of the connected components of $S \setminus D$, so $g^{-1}\gamma(y) \in S^-$ and $g\gamma(x) \in S^+$. Then $g^{k-1}\gamma(x) \in S^+$, so $D$ intersects the curve $g^{-1}\gamma([y, 1]) \cap g^{k-1}\gamma([0, x])$. Hence $C$ intersects either $g^{-1}C$ or $g^{k-1}C$, so $C$ intersects $gC$, contradicting the minimality of $k$.
\end{proof}

\begin{lemma}\label{lem:minimal-loop}
  Let $P$ be a self-intersecting bounding path with principal word $g$. If $g$ is parabolic or hyperbolic, then $P$ has a short self-intersecting segment.
\end{lemma}

\begin{proof}
  We may suppose via isometry that $P$ is a subpath of the bi-infinite bounding path stabilised by $g$. Let $a, b \in P$ be such that $\phi(b) = g^k\phi(a)$ for some positive integer $k$. Assume $a, b$ and $k$ are chosen such that $\mu(a, b)$ is minimal. Also assume, reversing $P$ if necessary, that $a \in [1, b]_T$. Note that $\mu(a, b) + \mu(b, ga) = \mu(1, g)$, so $\mu(a, b) \leq \mu(1, g)/2$. Let $Q = [a, b]_T$.
  It follows from \Cref{lem:intersection0} that $Q$ intersects $gQ$. Let $c \in Q, d \in Q \cap gQ$ be distinct such that $\phi(c) = \phi(d)$. Then $\mu(c, d) \leq \mu(a, b) + \mu(ga, b) \leq \mu(1, g)$, hence $[c, d]_T$ is short.
\end{proof}

The following lemmas relate intersections to distances between pairs of points in $\Hyp$. This guarantees the existence of a word whose displacement is shorter than the displacement of the generator it replaces.

\begin{lemma}\label{lem:metric-trick}
  Let $p, q, r, s \in \Hyp$ be such that $[p, q]$ intersects $[r, s]$ and $q \neq s$. Either $d(p, s) < d(p, q)$ or $d(r, q) < d(r, s)$.
\end{lemma}

\begin{proof}
  Suppose without loss of generality that $w \neq p$. Let $w \in [p, q] \cap [r, s]$.
  If $d(w, s) < d(w, q)$, then
  \[
    d(p, s) \leq d(p, w) + d(w, s) < d(p, w) + d(w, q) = d(p, q).
  \]
  If $d(w, q) < d(w, s)$, then 
  \[
    d(r, q) \leq d(r, w) + d(w, q) < d(r, w) + d(w, s) = d(r, s).
  \]
  If $d(w, s) = d(w, q)$, then
  \[
    d(p, s) < d(p, w) + d(w, s) = d(p, w) + d(w, q) = d(p, q).
  \]
  The strict inequality holds because otherwise $\{p, w, q, s\}$ is collinear, and both $q$ and $s$ lie on the side of $w$ opposite to $p$; this implies that $q = s$.
\end{proof}

\begin{lemma}\label{lem:metric-trick-2}
  Let $p, q, r, s \in \Hyp$ be such that $[p, q]$ intersects $[r, s]$ at some point $w$. If $q \neq s$ and $d(w, s) \leq d(w, q)$, then $d(p, s) < d(p, q)$.
\end{lemma}
\begin{proof}
  This follows by considering the cases $d(w, s) < d(w, q)$ and $d(w, s) = d(w, q)$ in the proof of \Cref{lem:metric-trick}.
\end{proof}

Suppose $x \in X^\pm$. If $g$ is a short word in $X$ such that $\abs{g} < \abs{x}$ and exactly one term in $g$ is either $x$ or $x^{-1}$, then $g$ is a \emph{good replacement} for $x$.

\begin{proposition}\label{prop:reduced-set-generates}
  If $g$ is a good replacement for $x \in X$, then $X' = X \setminus \{x\} \cup \{g\}$ generates $G$.
\end{proposition}

\begin{proof}
  Write $g$ as a word $x_1 \dots x_n$ in $X$. Let $i$ satisfy $x_i \in \{x, x^{-1}\}$. It suffices to show that $x_i \in \gen{X'}$. Indeed, $x_i = (x_{1}\dots x_{i-1})^{-1}g(x_{i+1}\dots x_n)^{-1}$.
\end{proof}

Our goal is to find a generating set of $G$ with ``sufficiently nice'' properties that we can gain information about the structure of the group. We make this precise with the following definition.

\begin{definition}\label{def:reduced}
  $X$ is \emph{reduced} if the following hold:
  \begin{enumerate}
    \item[A1.]
    $X \cap X^- = \varnothing$.
    \item[A2.]
    No short word in $X$ is elliptic.
    \item[A3.]
    No short word $g$ in $X$ is a good replacement for a term occurring in $g$. 
  \end{enumerate}
\end{definition}

If $G$ is discrete and torsion-free, then there is one case in which we cannot eliminate all self-intersections in $T$: where $G$ is not free, or equivalently the action of $G$ on $\Hyp$ is cocompact. As shown below, we can detect if this occurs.

\begin{lemma}\label{lem:half-loop-reduces}
  Suppose $S$ is a short self-intersecting segment of a leftmost path on $T$. Either $S$ is the boundary of a compact fundamental domain for $G$ and $X$ is a minimal generating set for $G$, or $X$ is not reduced.
\end{lemma}

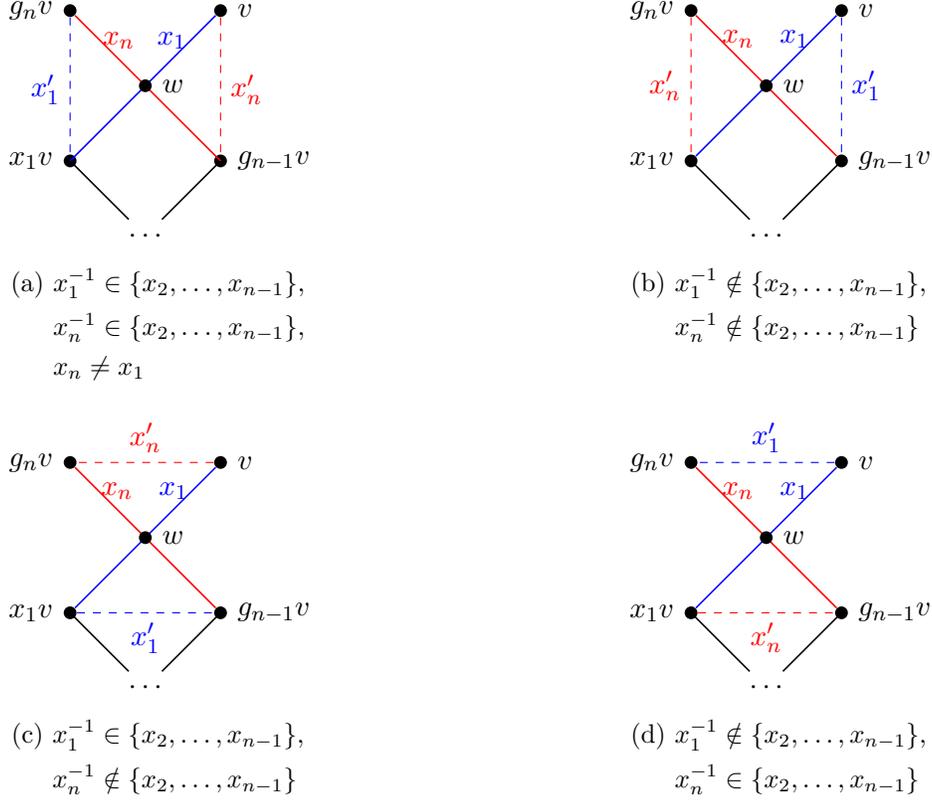
\begin{figure}[t]
  \centering
  \begin{subfigure}[t]{0.5\textwidth}
    \centering
    \begin{tikzpicture}
      \node[vertex, label={left:$g_nv$}] (gv) at (-1, 0) {};
      \node[vertex, label={right:$v$}] (v) at (1, 0) {};
      \node[vertex, label={left:$x_1v$}] (a) at (-1, -2) {};
      \node[vertex, label={right:$g_{n-1}v$}] (b) at (1, -2) {};
      \draw[edge, red] (gv) -- (b.center) node[pos=0.3, above] {$x_n$};
      \draw[edge, blue] (v) -- (a.center) node[pos=0.3, above] {$x_1$};
      \node (m) at (0, -3) {$\,\cdots$};
      \draw[edge] (b.center) -- (m) -- (a.center);
      \draw[dashed, blue] (gv) -- (a.center) node[midway, left] {$x_1'$};
      \draw[dashed, red] (v) -- (b.center) node[midway, right] {$x_n'$};
      \node[vertex, label={right:$w$}] (b) at (0, -1) {};
    \end{tikzpicture}
    \caption{
      $
      \begin{aligned}[t]
        &x_1^{-1} \in \{x_2, \dots, x_{n-1}\}, \\
        &x_n^{-1} \in \{x_2, \dots, x_{n-1}\}, \\
        &x_n \neq x_1
      \end{aligned}
      $
    }
  \end{subfigure}%
  \begin{subfigure}[t]{0.5\textwidth}
    \centering
    \begin{tikzpicture}
      \node[vertex, label={left:$g_nv$}] (gv) at (-1, 0) {};
      \node[vertex, label={right:$v$}] (v) at (1, 0) {};
      \node[vertex, label={left:$x_1v$}] (a) at (-1, -2) {};
      \node[vertex, label={right:$g_{n-1}v$}] (b) at (1, -2) {};
      \draw[edge, red] (gv) -- (b) node[pos=0.3, above] {$x_n$};
      \draw[edge, blue] (v) -- (a) node[pos=0.3, above] {$x_1$};
      \node (m) at (0, -3) {$\,\cdots$};
      \draw[edge] (b.center) -- (m) -- (a.center);
      \draw[dashed, red] (gv) -- (a) node[midway, left] {$x_n'$};
      \draw[dashed, blue] (v) -- (b) node[midway, right] {$x_1'$};
      \node[vertex, label={right:$w$}] (b) at (0, -1) {};
    \end{tikzpicture}
    \caption{
      $
      \begin{aligned}[t]
        &x_1^{-1} \notin \{x_2, \dots, x_{n-1}\}, \\
        &x_n^{-1} \notin \{x_2, \dots, x_{n-1}\}
      \end{aligned}
      $
    }
  \end{subfigure}%
  \par\bigskip
  \begin{subfigure}[t]{0.5\textwidth}
    \centering
    \begin{tikzpicture}
      \node[vertex, label={left:$g_nv$}] (gv) at (-1, 0) {};
      \node[vertex, label={right:$v$}] (v) at (1, 0) {};
      \node[vertex, label={left:$x_1v$}] (a) at (-1, -2) {};
      \node[vertex, label={right:$g_{n-1}v$}] (b) at (1, -2) {};
      \draw[edge, red] (gv) -- (b) node[pos=0.3, above] {$x_n$};
      \draw[edge, blue] (v) -- (a) node[pos=0.3, above] {$x_1$};
      \node (m) at (0, -3) {$\,\cdots$};
      \draw[edge] (b.center) -- (m) -- (a.center);
      \draw[dashed, blue] (a) -- (b) node[midway, below] {$x_1'$};
      \draw[dashed, red] (v) -- (gv) node[midway, above] {$x_n'$};
      \node[vertex, label={right:$w$}] (b) at (0, -1) {};
    \end{tikzpicture}
    \caption{
      $
      \begin{aligned}[t]
        &x_1^{-1} \in \{x_2, \dots, x_{n-1}\}, \\
        &x_n^{-1} \notin \{x_2, \dots, x_{n-1}\}
      \end{aligned}
      $
    }
  \end{subfigure}%
  \begin{subfigure}[t]{0.5\textwidth}
    \centering
    \begin{tikzpicture}
      \node[vertex, label={left:$g_nv$}] (gv) at (-1, 0) {};
      \node[vertex, label={right:$v$}] (v) at (1, 0) {};
      \node[vertex, label={left:$x_1v$}] (a) at (-1, -2) {};
      \node[vertex, label={right:$g_{n-1}v$}] (b) at (1, -2) {};
      \draw[edge, red] (gv) -- (b) node[pos=0.3, above] {$x_n$};
      \draw[edge, blue] (v) -- (a) node[pos=0.3, above] {$x_1$};
      \node (m) at (0, -3) {$\,\cdots$};
      \draw[edge] (b.center) -- (m) -- (a.center);
      \draw[dashed, red] (a) -- (b) node[midway, below] {$x_n'$};
      \draw[dashed, blue] (v) -- (gv) node[midway, above] {$x_1'$};
      \node[vertex, label={right:$w$}] (b) at (0, -1) {};
    \end{tikzpicture}
    \caption{
      $
      \begin{aligned}[t]
        &x_1^{-1} \notin \{x_2, \dots, x_{n-1}\}, \\
        &x_n^{-1} \in \{x_2, \dots, x_{n-1}\}
      \end{aligned}
      $
    }
  \end{subfigure}%
  \caption{Possible replacements of $x_i$ by $x_i'$, depending on the types of edges in $P$}
  \label{fig:replacements}
\end{figure}

\begin{proof}
  Suppose $X$ satisfies A1 and A2. Let $P$ be the minimal path containing $S$ (so that no subpath of $P$ contains $S$). Let $(g_0, \dots, g_n)$ be the vertices of $P$, and let $(x_1, \dots, x_{n-1})$ be the edge types. We may suppose $g_0 = 1$. If  some angle of $P$ is 0, then $x_i^{-1}v \in [v, x_{i+1}v]$ for some $i$, so $x_{i+1}x_i^{-1}$ is a good replacement for $x_{i+1}$. We thus assume otherwise.
  
  Note that each pair of edges of $P$ has at most one intersection.
  Hence $S$ has finitely many self-intersections, so we may assume $S$ is minimal, and in that case $S$ self-intersects only at its endpoints. Let $w \in \Hyp$ be the self-intersection point of $S$.
  
  First suppose $v = g_n v$, so $S = P$. In this case, the composition of $S$ with $\phi$ is a simple polygon. If there is some $i$ such that $x_i^{-1}$ does not occur in $g_n$, then $d(v, g_n v) = 0 < d(v, x_i v)$. Hence $g_n$ is a good replacement for $x_i$, so $X$ is not reduced. Suppose that $x_i^{-1}$ occurs in $g_n$ for all $i$. We show that every element of $X^\pm$ occurs once in $g_n$. Let $(y_1, \dots, y_{k})$ be the ordering of $X^\pm$ by $\measuredangle_v(y_1v, y_iv)$, such that $y_1 = x_1$. Because $P$ is leftmost, $y_{k+1}$ directly follows $(y_k)^{-1}$ in $g_n$ for each $k$. By induction, every element of $X^\pm$ occurs in $g_n$. By minimality, there are no self-intersecting subpaths of $P$. Thus the cosets of the region bounded by $S$ tile $\Hyp$. It follows from the Poincaré Polygon Theorem \cite[\textsection 7.4]{stillwell} that $X$ is a minimal generating set for $G$.
  
  Suppose now that $v \neq g_n v$. Note that $v$, $x_1 v$, $g_n v$, and $g_{n-1} v$ are all distinct, due to the minimality of $P$. First suppose $x_1 \neq x_n$. We replace some $x_j$ with $x_j'$, where $j \in \{1, n\}$ and
  \[
    x_j' \in \{g_n, g_{n-1}, x_1^{-1}g_n, x_1^{-1}g_{n-1}\}.
  \]
  \Cref{fig:replacements} shows the possible replacements depending on different conditions. In each case, the reduced word in $X$ representing $x_j'$ has exactly one term in $\{x_j, x_j^{-1}\}$. Additionally, \Cref{lem:metric-trick} shows that $\abs{x_j'} < \abs{x_j}$ for some $j \in \{1, n\}$. Thus $X$ violates A3 and is not reduced. Now suppose $x_1 = x_n$. Let $a$ and $b$ be the endpoints of $P$, and let $w = \phi(a) = \phi(b)$. If $x_1^{-1}$ does not occur in $g_n$, then case (b) of \Cref{fig:replacements} applies. Otherwise,
  \begin{align*}
    d(v, w) &= \mu(v, g_{n-1}v) - \mu(a, g_{n-1}v) \\
    &\geq \mu(a, b) - \mu(a, g_{n-1}v) \\
    &= d(g_{n-1}v, w).
  \end{align*}
  
  By \Cref{lem:metric-trick-2}, $d(x_1 v, g_{n-1}v) < d(v, x_1v)$, so $x_1^{-1}g_{n-1}$ is a good replacement for $x_1$.
\end{proof}

Using these lemmas we prove a key result.

\begin{proposition}\label{prop:reduced-is-discrete-free}
  If $X$ is reduced, then $G$ is torsion-free and $X$ is a minimal generating set for $G$.
\end{proposition}

\begin{proof}
  Suppose $X$ is reduced. If the action of $G$ on $\Hyp$ is cocompact and $X$ is a minimal generating set for $G$, then the result follows. Otherwise, by \Cref{lem:half-loop-reduces}, $T$ has no short self-intersecting segment. By \Cref{lem:minimal-loop} and \Cref{prop:path-impl-clock}, $T$ has no self-intersections, so $\phi$ is injective. Thus $\Cay(G, X)$ is a tree, so $G$ is free with basis $X$ (see \cite[I.3.2, Proposition 15]{serre}).
\end{proof}

The next key result is \Cref{cor:reduced-is-discrete}. To prove this, we determine conditions under which $T$ has a self-intersection.

Recall that a word $g$ in $X$ is \emph{cyclically reduced} if its length is at most the reduced word length of $xgx^{-1}$ for all $x \in X^\pm$ \cite[p.\ 9]{lyndon-schupp}.

\begin{lemma}\label{lem:elliptic-has-intersection}
  If a word $g$ in $X$ is elliptic, then $T$ has a self-intersection. If $g$ is cyclically reduced, then $\Set{g^k \given k \in \Z_{\geq 0}}$ is contained in a self-intersecting path on $T$.
\end{lemma}

\begin{proof}
  We may assume via conjugation that $g \in G$ is cyclically reduced, as every conjugate of $g$ is elliptic. It follows that $\Set{g^k \given k \in \Z_{\geq 0}}$ is contained in a path $S$ on $T$.
  
  If $v$ has a finite orbit under $\gen{g}$, then the desired statement holds, so suppose otherwise. Observe that $g$ stabilises a circle $C$ containing $v$ whose centre is the unique point fixed by $g$ (in particular, $g$ acts on $C$ by rotation), as shown in \Cref{fig:circle-intersect}. Let $A = [1, a]_T$ be the minimal segment of $[1, g]_T$ of nonzero length such that $\phi(a) \in C$ (a minimal segment exists because every edge intersects $C$ at most twice). Note that $\phi(A)$ is contained in the closure of a connected component $D$ of $\Hyp \setminus C$. Since $\phi(S) \cap C$ is dense in $C$, there exists $i \in \N$ such that $g^kv$ and $g^k\phi(a)$ lie on different connected components of $C \setminus \{v, \phi(a)\}$. Therefore $[1, a]_T$ intersects $[g^k, g^k a]_T$.
\end{proof}

\begin{figure}[t]
  \centering
  \begin{tikzpicture}
    \draw[edge] (0:2) arc (0:90:2);
    \draw[edge, dashed] (90:2) arc (90:360:2);
    \node[vertex, label={right:$v$}] (v) at (0:2) {};
    \node[vertex, label={[xshift=-5]above right:$\phi(a)$}] (a) at (90:2) {};
    \node[vertex, label={right:$gv$}] (gv) at (170:2) {};
    \draw[edge] (v) to[out=180, in=270] node[auto, pos=0.8]{$A$} (a);
    \draw[edge] (a) to[out=100, in=150, looseness=2] (gv);
    
    \node[vertex, label={right:$g^kv$}] (w) at (-20:2) {};
    \node[vertex, label={[yshift=5]right:$g^k\phi(a)$}] (gw) at (70:2) {};
    \draw[edge, blue] (w) to[out=160, in=250] node[auto, pos=0.2]{$g^kA$} (gw);
    
    \node at (45:1.4) {$D$};
    \node at (0,0) {$C$};
  \end{tikzpicture}
  \caption{Self-intersection of a path extending an elliptic word}
  \label{fig:circle-intersect}
\end{figure}
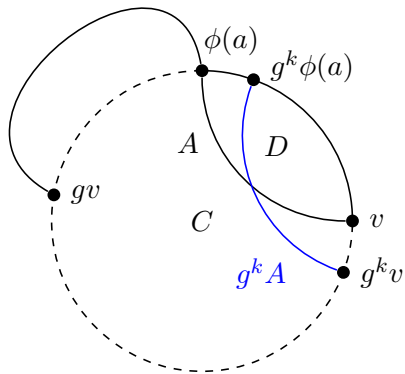

\begin{lemma}\label{lem:dense-line-has-intersection}
  If there exists some non-elliptic $g \in G$ such that $S = \Set{h \in H \given gh = hg \text{ and } \abs{h} < \abs{g}}$ is infinite, then $T$ has a self-intersection.
\end{lemma}

\begin{proof}
  Let $H = \gen{S}$. By compactness, $S$ has a limit point, so $Hv$ is a dense subset of a curve $L$ stabilised by $H$. Let $a \in [1, g]_T$ be such that $[1, a]_T \cap L$ is either $\{v, \phi(a)\}$ or a geodesic of nonzero length. Let $L'$ be the segment of $L$ bounded by $v$ and $\phi(a)$. Since $Hv$ is dense in $L$, there exists $x \in H$ such that $xv \in L'$, $xgv \notin L'$, and no vertex of $[x, xg]_T$ intersects a vertex of $[1, g]_T$. Thus $[x, xg]_T$ intersects the region bounded by $L'$ and $[1, a]_T$, so it intersects $[1, g]_T$.
\end{proof}

A subgroup of $\PSL_2(\R)$ is \emph{elementary} if there exists a finite orbit in $\Hyp$ or $\partial\Hyp$ \cite[\textsection 5.1]{beardon}. There are three types of elementary subgroups of $\PSL_2(\R)$:
\begin{enumerate}
  \item Groups that act by rotation around a common fixed point of $\Hyp$.
  \item Groups that act by a parabolic action fixing a common point of $\partial\Hyp$.
  \item Groups that stabilise a bi-infinite geodesic $\gamma$ of $\Hyp$. Every element of such a group acts either by translation along $\gamma$ or by a rotation of order 2 swapping the boundary points of $\gamma$.
\end{enumerate}
It is a special case of Beardon's classification \cite[\textsection 5.1]{beardon} of elementary subgroups of $\PSL_2(\C)$ that a discrete elementary subgroup $G$ of $\PSL_2(\R)$ is either cyclic or contains a cyclic subgroup $H$ of index 2. In the latter case, every element of $G \setminus H$ is elliptic of order 2.

\begin{lemma}\label{lem:cyclic-transitive}
  If $H, H' \in \PSL_2(\R)$ are cyclic subgroups with non-trivial intersection, then $\gen{H, H'}$ is cyclic.
\end{lemma}
\begin{proof}
  Two nontrivial elements of $\PSL_2(\R)$ commute if and only if they have the same fixed point set in $\Hyp \cap \partial\Hyp$. Hence every nontrivial element of $\gen{H, H'}$ has the same fixed point set. It follows that $\gen{H, H'}$ is isomorphic to a subgroup of the additive group $\R$, so it is cyclic.
\end{proof}

\begin{proposition}\label{prop:not-disc-free-has-intersection}
  If $G$ is indiscrete, then $T$ has a self-intersection.
\end{proposition}

\begin{proof}
  \Cref{lem:elliptic-has-intersection} covers the case where $G$ has elliptic elements, so we suppose that no element of $G$ is elliptic. By \cite[Theorem 8.2.6]{beardon}, $G$ is elementary.
  
  Suppose for a contradiction that every pair of elements in $X$ generate a discrete group. Let $S$ be a maximal subset of $X$ such that $\gen{S}$ is discrete. Let $a \in X \setminus S$, $b \in S$. Note that $\gen{S}$ and $\gen{a, b}$ are both discrete and elementary, and thus cyclic. By \Cref{lem:cyclic-transitive}, the group $H = \gen{S, a}$ is cyclic. Since $H$ is generated by a parabolic or hyperbolic element, it is discrete. By contradiction, there must exist $g, h \in X$ such that $\gen{g, h}$ is discrete.
  
  If $g$ and $h$ commute, then $ghg^{-1}h^{-1}v = v$, so $T$ has a self-intersection. Suppose $g$ and $h$ do not commute. We may assume that $g$ is hyperbolic. By \cite[Theorem 4.3.5]{beardon}, the commutators $c = ghg^{-1}h^{-1}$ and $c' = cgc^{-1}g^{-1}$ are parabolic and commute. Therefore $cc'c^{-1}c'^{-1}v = v$, so $T$ has a self-intersection.
\end{proof}

\begin{corollary}\label{cor:reduced-is-discrete}
  If $X$ is reduced, then $G$ is discrete.
\end{corollary}

\begin{proof}
  If $X$ is reduced then, by \Cref{lem:half-loop-reduces}, either the action of $G$ on $\Hyp$ is cocompact (hence $G$ is discrete) or $T$ has no short self-intersecting segments. As in the proof of \Cref{prop:reduced-is-discrete-free}, the latter case implies that $T$ has no self-intersection. By \Cref{prop:not-disc-free-has-intersection}, $G$ is discrete.
\end{proof}

\subsection{A bound on the displacement of a replacement}

Define $\Phi \colon [0, \infty) \to [0, \infty)$ by
\[
  \Phi(x) = \sinh\left(\frac{x}{2}\right).
\]
The following theorem is due to Beardon \cite[Theorem 8.3.1]{beardon}.

\begin{theorem}\label{prop:collar}
  Let $g, h \in \PSL_2(\R)$. If $\gen{g, h}$ is discrete, torsion-free, and non-abelian, then $\Phi(\abs{g})\Phi(\abs{h}) \geq 1$.
\end{theorem}
Note that $\Phi$ is monotonically increasing, $\Phi(0) = 0$, and $\lim_{x \to \infty}\Phi(x) = 1$. This implies the following:

\begin{corollary}\label{cor:non-commuting-bound}
  Suppose $G$ is discrete and torsion-free. For all $\e > 0$, if $g, h \in X$ do not commute and $\abs{g} \leq \Phi^{-1}(1/\Phi(\e))$, then $\abs{h} \geq \e$.
\end{corollary}

A core procedure in our recognition algorithm iteratively replaces generators with words of smaller displacement until a short word is found with displacement below some $\delta(X)$ determined by \Cref{cor:non-commuting-bound}. If $G$ is discrete, then only finitely many such replacements can occur. However, if $G$ is indiscrete, then the displacements could remain bounded from below by $\delta$ at every step, so the algorithm would not terminate.
The following proposition eliminates this scenario.

\begin{restatable}{proposition}{propbound}
  \label{prop:indiscrete-reduce-bound}
  Let $X$ be a finite subset of $\PSL_2(\R)$ generating a group $G$. Let $\e > 0$.
  One of the following holds:
  \begin{enumerate}
    \item $G$ is discrete and torsion-free.
    \item There exists an elliptic short word $h$ of $X$.
    \item There exists a short word $h$ of $X$ such that $\abs{h} < \e$.
    \item There exists a good replacement $g$ for some $x \in X$ such that $\abs{x}-\abs{g} \geq \delta(\abs{x}, \e)$, where
    \begin{gather*}
      \delta(r, \e) = \min\{\e, r - \cosh^{-1}(\cosh(r) - \eta(r, \e))\} > 0, \\
      \eta(r, \e) = 2\sinh^3\left(\frac{\e}{2}\right)\csch\left(r+\frac{\e}{2}\right).
    \end{gather*}
  \end{enumerate}
\end{restatable}
(Here $\csch$ refers to hyperbolic cosecant.) For a proof, see \Cref{ap:bound}.

\section{The recognition algorithm}

\begin{lemma}\label{lem:decide-commute}
  There exists an algorithm that decides whether a pair $g, h$ of commuting non-elliptic elements of $\PSL_2(\R)$ generates a discrete group, and if so outputs a generator of $\gen{g, h}$.
\end{lemma}
\begin{proof}
  See \cite[\textsection 3]{kirschmer}.
\end{proof}

\begin{theorem}\label{recognition-torsion-free}
  Let $X$ be a finite subset of $\SL_2(\R)$ or $\PSL_2(\R)$. There exists an algorithm that decides whether the group $G$ generated by $X$ is discrete and torsion-free, and if so outputs a minimal generating set for $G$.
\end{theorem}

\afterpage{%
\vspace*{\fill}
\begin{algorithm}
  \KwData{Finite subset $X$ of $\PSL_2(\R)$}
  \KwOut{Either (\texttt{False}, $g$) where $g$ is elliptic; (\texttt{False}, ($g$, $h$)) where $\gen{g, h}$ is indiscrete; or (\texttt{True}, $X'$) where $G = \gen{X}$ is discrete and torsion-free and $X'$ is a reduced generating set of $G$}
  $X' \gets X$ \;
  \Loop{
    \uIf{there exists $x \in X'$ such that $x^{-1} \in X'$}{
      $X' \gets X' \setminus \{x\}$ \;
      \Continue\;
    }
    \uElseIf{$\abs{X'} = 0$}{
      \Return (\texttt{True}, $\varnothing$)\;
    }
    \ElseIf{$\abs{X'} = 1$}{
      \uIf{$x \in X'$ is elliptic}{\Return (\texttt{False}, $g$)\;}
      \Else{\Return (\texttt{True}, $X'$)\;}
    }
    let $a, b \in X'$ be distinct such that $\abs{a} \leq \abs{b} \leq \abs{x}$ for all $x \in X' \setminus \{a, b\}$\;
    \uIf{$\gen{a, b}$ is abelian and indiscrete}{
      \Return (\texttt{False}, $\{a, b\}$)\;
    }
    \uElseIf{$\gen{a, b} = \gen{c}$ for some $c$}{
      $X' \gets X' \setminus \{a, b\} \cup \{c\}$\;
    }
    \uElseIf{there exists an elliptic short word $g$ of $X'$}{
      \Return (\texttt{False}, $g$)\;
    }
    \Else{
      let $g$ be a short word in $X'$ such that $\abs{g}$ is minimal\;
      \If{$\Phi(\abs{g})\Phi(\abs{b}) < 1$ and $g$ is non-trivial in $G$}{
        choose $c \in \{a, b\}$ such that $cg \neq gc$\;
        \Return (\texttt{False}, $\{c, g\}$)\;
      }
      $S \gets \Set{(x, h) \given x \in X',\ h \text{ is a short word in } X'\text{ reducing } x}$\;
      \uIf{$S \neq \varnothing$}{
        let $(x, h) \in S$ be such that $\abs{x} - \abs{h}$ is minimal\;
        $X' \gets X' \setminus \{x\} \cup \{h\}$\;
      }
      \Else{
        \Return (\texttt{True}, $X'$)\;
      }
    }
  }
  \caption{Decide whether a subgroup of $\PSL_2(\R)$ is discrete and torsion-free}\label{alg:free}
\end{algorithm}
\vspace*{\fill}
\pagebreak}

\Cref{alg:free} establishes \Cref{recognition-torsion-free}. Each iteration either applies a single transformation to $X$ or returns \texttt{True} or \texttt{False}. We outline the procedure run at each iteration.
\begin{enumerate}
  \item If $X$ violates A1, then remove an element of $X \cap X^-$ from $X$.
  \item Handle $\abs{X} \le 1$ as a special case.
  \item Let $a$ and $b$ be the elements of $X$ with smallest and second-smallest (possibly equal) displacement respectively. If $\gen{a, b}$ is abelian and indiscrete, then return \texttt{False}. If $\gen{a, b}$ is cyclic, then replace $a$ and $b$ with a single generator. 
  \item If $X$ violates A2, then return \texttt{False}.
  \item Find a short word $g$ in $X$ such that $\abs{g}$ is minimal. If $\Phi(\abs{g})\Phi(\abs{b}) < 1$, then return \texttt{False}.
  \item If no short word in $X$ replaces a generator, then return \texttt{True}. Otherwise, find a short word $h$ in $X$ reducing some term $x$ such that $\abs{x} - \abs{h}$ is minimal. Replace $x$ in $X$ by $h$. 
\end{enumerate}
These steps are repeated until the algorithm returns. We now prove that this algorithm is correct and terminates.

\begin{proof}
  We consider the possible results of each transformation. Let $X_i$ be the value assigned to $X'$ during the $i$-th iteration of the loop. We may assume that $\abs{X_i} \geq 2$. If $X_i \cap X_i^{-1} \neq \varnothing$, then $\abs{X_{i+1}} < \abs{X_i}$.
  Let $a_i, b_i \in X_i$ be such that $\abs{a_i} \leq \abs{b_i} \leq \abs{x}$ for all $x \in X_i$.
  If $\Phi(\abs{a_i})\Phi(\abs{b_i}) < 1$ then, by \Cref{cor:non-commuting-bound}, either $H = \gen{a_i, b_i}$ is indiscrete or $a_i$ and $b_i$ commute. By \Cref{lem:decide-commute}, we can decide whether $H$ is discrete and, if so, find $c_i \in G$ generating $H$. We thus replace $a_i$ and $b_i$ by $c_i$, in which case $\abs{X_{i+1}} < \abs{X_i}$.
  
  We may therefore assume that $a_i$ and $b_i$ do not commute. We may also suppose there is no short elliptic word of $X_i$.
  Suppose there is a short word $g$ such that $\abs{g} < \e_i$. Observe that $g$ does not commute with some $c \in \{a_i, b_i\}$, since $a_i$ and $b_i$ do not commute and therefore do not have the same set of eigenspaces. By \Cref{cor:non-commuting-bound}, $\gen{g, c}$ is indiscrete.
  
  If there is a good replacement $g$ for some $x_i \in X_i$, then we replace $x_i$ by $g$. It follows that $s_{i+1} < s_i$, where $s_i = \sum_{x \in X_i}\abs{x}$. If the algorithm terminates on some $X_n$ without returning an elliptic element or an indiscrete subgroup, then no short words in $X_n$ reduce a generator of $X_n$, so $X_n$ is reduced. By \Cref{prop:reduced-is-discrete-free} and \Cref{cor:reduced-is-discrete}, $X_n$ is discrete and torsion-free.
  
  We now show that the algorithm terminates. For a contradiction, suppose otherwise. For some $k \in \N$, it must be that $\abs{X_{i}}$ is constant for all $i > k$, so $X_{i+1}$ is obtained by replacing some $x_i \in X_i$ with a good replacement $g_i$ for $x_i$.

  Let $M_i$ and $m_i$ be the maximum and minimum respectively of $\Set{\abs{x} \given x \in X}$.
  Each $g_i$ satisfies $\abs{g_i} < M_i$. Since infinitely many of the $g_i$ are distinct it follows that $G$ is indiscrete.
  It also follows that $M_i \geq \abs{y} \geq m_i$ for each $y \in X_i$. Since $m_{i+1} \geq m_i$ for each $i$, we see that $m_i$ converges to some $m > 0$ and $\abs{y} \geq m$ for all $y \in X_j$ and $j \geq i$. Hence the sequence $(s_j)_{j \in \N}$ is decreasing and bounded below by $\abs{X_i}m$, and so converges. However, by \Cref{prop:indiscrete-reduce-bound},
  \begin{align*}
    \lim_{j \to \infty} (s_{j+1} - s_{j}) &\geq \lim_{j \to \infty} \min\Set{\abs{x} - \abs{g} \given g\text{ is a good replacement for }x \in X_j} \\
    &\geq \lim_{j \to \infty} \min_{x \in X_j}\delta(m, \abs{x}) \\
    &> 0.
  \end{align*}
  By contradiction, the algorithm must terminate.
\end{proof}

\subsection{Performing exact computations}\label{sec:exact}
\Cref{alg:free} assumes that it is possible to perform exact operations and comparisons with real numbers (for example, by an oracle that can decide whether $x < y$ for real numbers $x$ and $y$). We show that the only operations needed are field operations; in particular, if the entries are algebraic, then the algorithm can be performed with exact precision. We use the following statement from \cite[Theorem 7.2.1]{beardon}.

\begin{theorem}\label{thm:tanh-dist}
  If $w$ and $z$ are complex numbers and $d$ is the distance metric in the upper half plane model of $\Hyp$, then
  \[
    \tanh\left(\frac{d(w, z)}{2}\right) = \abs*{\frac{z-w}{z-\overline{w}}},
  \]
  where $\overline{w}$ is the complex conjugate of $w$.
\end{theorem}

Since $\tanh$ is monotonically increasing, $d(w, z) < d(w, z')$ if and only if
\[
  \abs*{\frac{z-w}{z-\overline{w}}}^2 < \abs*{\frac{z'-w}{z'-\overline{w}}}^2.
\]
\Cref{alg:free} also uses the real-valued function $\Phi$ to decide whether $\Phi(\abs{g})\Phi(\abs{h}) < 1$ for some $g, h \in \PSL_2(\R)$.

\begin{proposition}
  If $g, h \in \PSL_2(\R)$, then $\Phi(\abs{g})\Phi(\abs{h}) < 1$ if and only if
  \[
  \abs*{\frac{gv - v}{gv - \overline{v}}}^2 + \abs*{\frac{hv - v}{hv - \overline{v}}}^2 < 1.
  \]
\end{proposition}
\begin{proof}
  Let $x = \abs{g}/2$, $y = \abs{h}/2$. Then
  \[
    \tanh^2(y) = \frac{1}{1 + \csch^2(y)}
  \]
  and
  \[
    1 - \tanh^2(x) = \sech^2(x) = \frac{1}{1+\sinh^2(x)}.
  \]
  Additionally, $\Phi(\abs{g})\Phi(\abs{h}) < 1$ if and only if $\sinh(x) < \csch(y)$. This holds if and only if
  \[
    \frac{1}{1+\sinh^2(x)} > \frac{1}{1+\csch^2(y)},
  \]
  which in turn holds if and only if $\tanh^2(x) + \tanh^2(y) < 1$. Applying \Cref{thm:tanh-dist} completes the proof.
\end{proof}

As these are the only real computations performed in \Cref{alg:free}, the algorithm runs with exact precision over a real algebraic number field.

\section{Properties of reduced sets}

From a reduced subset $X$ of $\PSL_2(\R)$, we can determine various geometric invariants of $G = \gen{X}$.

\subsection{Bounding polygons}
Let $P$ be a leftmost path. The set $\Hyp \setminus \phi(P)$ has two connected components. One is an infinite-sided polygon that has as its interior angles the clockwise angle around each vertex of $P$; the closure of this component is the \emph{bounding polygon} of $T$ with boundary $P$. Note that a bounding polygon with boundary $P$ is isometric in $G$ to a bounding polygon stabilised by the principal word of $P$.

When convenient, we identify a bounding polygon with its closure in $\Hyp \cap \partial \Hyp$. A bounding polygon $D$ intersects $\partial\Hyp$ at an arc $\sigma$ (which is possibly a single point). If $g$ stabilises $D$, then it also stabilises $\sigma$, so the endpoints of $\sigma$ are the limit points of $g$.

\begin{proposition}\label{bounding-polygon}
  Suppose $X$ is reduced. Every connected component of $\Hyp \setminus \phi(T)$ is the interior of a bounding polygon of $T$. Conversely, every bounding polygon of $T$ is the closure of a connected component of $\Hyp \setminus \phi(T)$.
\end{proposition}
\begin{proof}
  Let $D$ be a bounding polygon of $T$ with boundary $P = (g_i)_{i \in \Z}$. Suppose for a contradiction that there exists some $a \in T \cap D^\circ$, where $D^\circ$ is the interior of $D$. Since $X$ is reduced, two cases occur. If the action of $G$ on $\Hyp$ is cocompact then, by \Cref{lem:half-loop-reduces}, every bounding polygon is a coset of $D$ and a connected component of $\Hyp \setminus \phi(T)$. Otherwise assume that the action of $G$ on $\Hyp$ is not cocompact. As in the proof of \Cref{prop:reduced-is-discrete-free}, $T$ has no self-intersections. Hence there is some $i \in \Z$ such that the image of $[a, g_i]_T \setminus \{g_i\}$ under $\phi$ is contained in $D^\circ$. Let $h$ be the vertex adjacent to $g_i$ in the minimal path containing $[a, g_i]_T$. Then $B(g_i^{-1}g_{i-1}, g_i^{-1}h, g_i^{-1}g_{i+1})$ holds. But $P$ is leftmost, a contradiction.
  
  Conversely, let $D$ be a connected component of $\Hyp \setminus \phi(T)$. Let $P$ be the path bounding $D$, such that the vertex sequence of $P$ is in clockwise order. Given an edge $e$ of $P$, the edge of $\lmost(e)$ adjacent to $e$ either coincides with $P$ or intersects $D^\circ$. The latter cannot hold, so $P$ is a leftmost path.
\end{proof}

The bounding polygons are in 1-1 correspondence with the set of conjugates of principal words. In this way, we can enumerate the bounding polygons of $C$.

\subsection{The Nielsen region}
The set $\partial\Hyp$ is the disjoint union of the limit set of $G$ and a countable union of pairwise disjoint open arcs $\sigma_i$. Let $L_i$ be the geodesic whose boundary points are the endpoints of $\sigma_i$. The \emph{Nielsen region} $N$ is the open region of $\Hyp$ bounded by the set of all $L_i$ \cite[\textsection 8.5]{beardon}. Equivalently, $N$ is the smallest $G$-invariant open convex subset of $\Hyp$ \cite[Theorem 8.5.2]{beardon}.

\Cref{bounding-polygon} shows that each $\sigma_i$ lies in a bounding polygon. Given a bounding polygon with principal word $g$, the endpoints of the associated $\sigma_i$ is the limit set of $g$. Therefore each $L_i$ is determined by the limit set of a conjugate of a principal word; thus we can algorithmically list the sides of the Nielsen region.

\subsection{The signature}
Suppose $G$ is discrete with genus $g$. An \emph{interval of discontinuity} is a maximal interval $\sigma \subset \partial\Hyp$ such that for all $g \in G$, either $g\sigma = \sigma$ or $g\sigma \cap \sigma$ has nonempty interior. A \emph{boundary} hyperbolic element of $G$ leaves some interval of discontinuity on $\partial\Hyp$ invariant. A cyclic subgroup of $G$ is \emph{elliptic}, \emph{parabolic}, or \emph{boundary hyperbolic} if its nontrivial elements have the respective property.

A finitely generated Fuchsian group has a finite number of conjugacy classes of maximal cyclic subgroups that are cyclic, parabolic, or boundary hyperbolic \cite[\textsection 10.3]{beardon}. If, of these conjugacy classes,
\begin{itemize}
  \item $r$ are elliptic with orders $m_1, \dots, m_r$;
  \item $s$ are parabolic; and
  \item $t$ are boundary hyperbolic;
\end{itemize}
then the tuple $(g: m_1, \dots, m_r \,;\, s \,;\, t)$ is the \emph{signature} of $G$ \cite[\textsection 10.4]{beardon}.

\begin{proposition}
  If $X$ is reduced, then every parabolic and boundary hyperbolic element of $G$ is conjugate to a principal word in $X$. Conversely, every principal word in $X$ is either parabolic or boundary hyperbolic.
\end{proposition}
\begin{proof}
  Let $g$ be a parabolic or boundary hyperbolic element of $G$. Assume via conjugation that $g$ is cyclically reduced. By \Cref{bounding-polygon}, the bi-infinite path on $T$ containing the points $\Set{g^k \given k \in \Z}$ bounds a polygon $D$ in $\Hyp$ intersecting $\partial\Hyp$ on some arc $\sigma$. Let $D_1$ and $D_2$ be bounding polygons contained in $D$, with principal words $h_1$ and $h_2$ respectively. The limit points of $h_1$ and $h_2$ are subsets of $L$. If $g$ is parabolic, then $h_1$ and $h_2$ are also parabolic with the same limit point. If $g$ is hyperbolic then, since $G$ acts discontinuously on the interior of $\sigma$, the boundary points of $h_1$ and $h_2$ are also boundary points of $g$. By \cite[Theorem 5.1.2]{beardon}, the set of boundary points of $h_1$, $h_2$, and $g$ coincide. In either case $D_1 = D_2 = D$, hence $D$ is a bounding polygon and $g$ is conjugate to a principal word in $X$.
  
  Conversely, let $g$ be a principal word in $X$. Suppose $g$ is not parabolic. Then $g$ stabilises an arc $\sigma$ in $\partial\Hyp$ with non-empty interior. Since $\sigma$ is contained in the bounding polygon stabilised by $g$, the interior of $\sigma$ has no limit point of $G$, so $g$ is boundary hyperbolic.
\end{proof}
We can therefore compute the signature of a finitely generated torsion-free Fuchsian group by listing the conjugacy classes of principal words (for which a set of representatives is a subset of the principal words found in Algorithm \ref{alg:free}) and checking how many are parabolic.

\section{The constructive membership problem}\label{sec:membership}

Here we provide an algorithm to solve the constructive membership problem for finitely generated discrete torsion-free subgroups of $\SL_2(\R)$. We proceed similarly to \cite{markowitz}. Whereas in \Cref{alg:free} we apply a reduction algorithm to the generating set, here we keep the generating set fixed and instead apply a reduction algorithm to a given point of $\Hyp$.


\begin{lemma}\label{lem:intersection1}
  Suppose $X$ is reduced. Let $g$ be a principal word of some infinite bounding path $P$ of $T$. Let $a \in P$ be such that $[1, a]_T$ is not short. If $C$ is a simple curve in $\Hyp$ with endpoints $v$ and $\phi(a)$, then $C$ intersects either $gC$ or $[1, g]_T \setminus \{1\}$.
\end{lemma}
\begin{proof}
  Suppose $C$ does not intersect $gC$. If a segment $[1, b]_T$ of $P$ of nonzero length maps under $\phi$ to a subset of $C$, then the proof is complete, so suppose otherwise. We may assume (choosing $a$ and restricting $C$ if necessary) that $[1, a]_T$ intersects $C$ only at $1$ and $a$. Let $k \in \Z_{\geq 0}$ be maximal such that $g^k \in [1, a]_T \setminus \{a\}$.
  
  Let $\psi \colon \Hyp \to \Hyp$ be a continuous $g$-equivariant function that maps $g^iC$ to a unique point for all $i \in \Z$, and is injective on all other points. This map is well-defined since, by \Cref{lem:intersection0}, the curves of the form $g^kC$ are pairwise disjoint.
  
  Let $\gamma$ be the restriction of $\psi \circ \phi$ to $[1, a]_T$.
  Note that $\gamma$ intersects $g^k\gamma$. By the proof of \Cref{lem:minimal-loop}, $\gamma$ has a self-intersection at some points $x, y \in [1, a]_T$. Since $[1, a]_T$ cannot self-intersect, it must hold that $x$ and $y$ are both elements of $g^iC$ for some $i \in \Z$. However $[1, a]_T$ only intersects such curves twice, these curves being $C$ and $g^kC$. Therefore $k=0$ and $C$ intersects $[1, g]_T \setminus \{1\}$.
\end{proof}

\begin{lemma}\label{lem:shorter-from-intersection}
  Suppose $X$ is reduced. Let $w \in \Hyp$. If $[v, w] \setminus \{v\}$ intersects a bounding path $P$ starting from $v$ with principal word $g$, then there exists a subword $h$ of $g$ or $g^{-1}$ such that $d(hv, w) < d(v, w)$.
\end{lemma}

\begin{figure}[t]
  \centering
  \begin{subfigure}[t]{0.5\textwidth}
    \centering
    \begin{tikzpicture}
      \tkzDefPoint(0,1){v1}
      \tkzDefPoint(0,0){v2}
      \tkzDefPoint(1,-1){v3}
      \tkzDefPoint(2,-1){v4}
      \tkzDefPoint(2.5,1){v5}
      
      \tkzDefPoint(3,0){w}
      \tkzInterLL(v1,w)(v4,v5) \tkzGetPoint{a}
      
      \tkzDrawPolySeg[edge](v1,v2,v3,v4,v5)
      \tkzDrawSegment[edge, dashed](v1,w)
      
      \tkzDrawPoint[vertex](v1) \tkzLabelPoint[left](v1){$v$}
      \tkzDrawPoint[vertex](v4) \tkzLabelPoint[right](v4){$h$}
      \tkzDrawPoint[vertex](v5) \tkzLabelPoint[right](v5){$h'$}
      
      \tkzDrawPoint[vertex](a) \tkzLabelPoint[below left](a){$\phi(a)$}
      \tkzDrawPoint[vertex](w) \tkzLabelPoint[right](w){$w$}
    \end{tikzpicture}
    \caption{$P$ is short}
  \end{subfigure}%
  \begin{subfigure}[t]{0.5\textwidth}
    \centering
    \begin{tikzpicture}
      \tkzDefPoint(0,1){v1}
      \tkzDefPoint(0,0){v2}
      \tkzDefPoint(1,-1){v3}
      \tkzDefPoint(2,-1){v4}
      \tkzDefPoint(2.5,1){v5}
      
      \tkzDefPoint(3,0){w}
      \tkzInterLL(v1,w)(v4,v5) \tkzGetPoint{a}
      
      \tkzDefPoint(1.5,1.8){ga}
      \tkzInterLL(v1,w)(v4,ga) \tkzGetPoint{p}
      
      \tkzDrawPolySeg[edge](v1,v2,v3,v4,v5)
      \tkzDrawSegment[edge, dashed](v1,w)
      \tkzDrawSegment[edge, dashed](v4,ga)
      
      \tkzDrawPoint[vertex](v1) \tkzLabelPoint[left](v1){$v$}
      \tkzDrawPoint[vertex](v4) \tkzLabelPoint[right](v4){$gv$}
      
      \tkzDrawPoint[vertex](a) \tkzLabelPoint[above right](a){$\phi(a)$}
      \tkzDrawPoint[vertex](w) \tkzLabelPoint[right](w){$w$}
      \tkzDrawPoint[vertex](ga) \tkzLabelPoint[above](ga){$gw$}
      \tkzDrawPoint[vertex](p) \tkzLabelPoint[below left](p){$p$}
    \end{tikzpicture}
    \caption{$P$ is not short}
  \end{subfigure}
  \caption{Diagrams for \Cref{lem:shorter-from-intersection}}
  \label{fig:shorter-from-intersection}
\end{figure}

\begin{proof}
  Assume $P$ is minimal. Let $a \in P \cap [v, w] \setminus \{1\}$ be such that $[1, a]_T \cup [v, \phi(a)]$ is the boundary of a simply connected closed subset $D$ of $\Hyp$. Suppose $P$ is short, as in \Cref{fig:shorter-from-intersection} (a). Let $e$ be the edge of $P$ containing $a$. Let $x$ be the type $e$, and let $\{h, h'\}$ be the vertices of $e$ so that exactly one term of $h$ is either $x$ or $x^{-1}$. If $w = h'v$, then $d(h'v, w) < d(v, w)$ as desired, so suppose otherwise. From A3 it follows that $d(v, h'v) \geq d(v, xv) = d(hv, h'v)$ so, by \Cref{lem:metric-trick}, $d(w, hv) < d(w, v)$.
  
  Now suppose $P$ is not short. By \Cref{lem:intersection1}, $[v, w]$ intersects either $[gv, gw]$ or $[1, g]_T$. The latter contradicts minimality of $P$, so the former holds, as shown in \Cref{fig:shorter-from-intersection} (b). By \Cref{lem:metric-trick}, either $d(gv, w) < d(v, w)$ or $d(g^{-1}v, w) = d(v, gw) < d(gv, gw) = d(v, w)$.
\end{proof}

\begin{proposition}\label{prop:short-words-closer}
  Suppose $X$ is reduced. Let $w \in \Hyp$ and let $g$ be a reduced word in $X$.
  \begin{enumerate}
    \item If $d(gv, w) < d(v, w)$, then there exists a short word $h$ such that $d(hv, w) < d(v, w)$.
    \item If $g$ is not short and $d(gv, w) \leq d(g'v, w)$ for all $g' \in G$, then $d(gv, w) < d(v, w)$.
  \end{enumerate}
\end{proposition}

\begin{proof}
  \Cref{cor:reduced-is-discrete} shows that $G$ is discrete, so we may assume that $d(gv, w) \leq d(g'v, w)$ for all $g' \in G$. Let $P=[1, g]_T$. If $g$ is short and $d(gv, w) < d(v, w)$, then we take $h = g$. We therefore suppose $g=x_1\dots x_n$ is not short and $d(gv, w) \leq d(v, w)$.
  
  We construct the diagrams shown in \Cref{fig:membership}. Let $Q$ be a bi-infinite bounding path on $T$ whose bounding polygon $D$ contains an initial segment of $[v, w]$. Similarly, let $Q'$ be a bi-infinite bounding path on $T$ whose bounding polygon $D'$ contains an initial segment of $[gv, w]$. Since $X$ is reduced, by \Cref{bounding-polygon} either $D = D'$ or the interiors $D^\circ$ and $D'^\circ$ are disjoint. First suppose $D \neq D'$. Since $D^\circ \cap D'^\circ = \varnothing$, one of the following cases must hold:
  \begin{enumerate}
    \item Suppose $w \notin D^\circ$, as shown in \Cref{fig:membership} (a). Then $Q$ intersects $[v, w] \setminus \{v\}$ so, by \Cref{lem:shorter-from-intersection}, there exists a short word $h$ such that $d(hv, w) < d(v, w)$.
    \item Suppose $w \notin D'^\circ$, as shown in \Cref{fig:membership} (b). By multiplying by $g^{-1}$ and applying case (1), we find that there exists a short word $h$ such that $d(hv, g^{-1}w) < d(v, g^{-1}w)$. Equivalently, $d(ghv, w) < d(gv, w)$, which is a contradiction.
  \end{enumerate}
  Secondly, suppose $D = D'$. The closure of $D$ contains a segment of $[v, w] \cup [w, gv]$ containing both $v$ and $gv$, so $[v, w] \cup [w, gv]$ is in the closure of $D$. It follows from \Cref{lem:intersection1} that $[v, w] \cup [w, gv]$ intersects either $[\sigma v, \sigma w] \cup [\sigma w, \sigma gv]$ or $[1, \sigma]_T$, where $\sigma$ is the principal word of $Q$; the latter implies $g \in [1, \sigma]_T$, contradicting the assumption that $g$ is short. We apply \Cref{lem:metric-trick} to each of the following cases.
  \begin{enumerate}
    \item
    If $[v, w]$ intersects $[\sigma v, \sigma w]$, then either $d(\sigma^{-1}v, w) = d(v, \sigma w) < d(\sigma v, \sigma w) = d(v, w)$ or $d(\sigma v, w) < d(v, w)$.
    \item 
    If $[w, gv]$ intersects $[\sigma w, \sigma gv]$, then as in (1) either $d(\sigma^{-1}gv, w) < d(gv, w)$ or $d(\sigma gv, w) < d(gv, w)$; each is a contradiction.
    \item
    If $[v, w]$ intersects $[\sigma w, \sigma gv]$ then, since $d(\sigma gv, w) \geq d(gv, w)$, we conclude that $d(\sigma^{-1}v, w) = d(v, \sigma w) < d(\sigma gv, \sigma gw) = d(v, w)$.
    \item
    If $[\sigma v, \sigma w]$ intersects $[w, gv]$, then as in case (3) we find that $d(\sigma v, w) < d(v, w)$.
    \qedhere
  \end{enumerate}
\end{proof}

\begin{figure}[t]
  \centering
  \begin{subfigure}[t]{0.5\textwidth}
    \begin{tikzpicture}
      \node[vertex, label=left:$v$] (v) at (0, 0) {};
      \node[vertex, label=below right:$w$] (w) at (4.8, 0) {};
      \node[vertex, label=right:$gv$] (gv) at (5, 1) {};
      \coordinate (x1) at (1, -1.5);
      \coordinate (x2) at (1.8, -0.6);
      \coordinate (x3) at (2.2, 0.8);
      \coordinate (y1) at (3.5, 0.8);
      \coordinate (y2) at (3, -1.4);
      \draw[edge] (v) -- (w) -- (gv);
      \draw[edge, blue] (v) -- (x1) -- (x2) -- (x3);
      \draw[edge, blue] (gv) -- (y1) -- (y2);
      
      \tkzLabelLine[blue, left, pos=0.8](x2, x3){$Q$}
      \tkzLabelLine[blue, left, pos=0.8](y1, y2){$Q'$}
      
      \node at (1.05, -0.6) {$D$};
      \node at (4.2, 0.4) {$D'$};
    \end{tikzpicture}
    \caption{$w \notin D^\circ$}
  \end{subfigure}%
  \begin{subfigure}[t]{0.5\textwidth}
    \begin{tikzpicture}
      \node[vertex, label=left:$g^{-1}v$] (v) at (0, -0.5) {};
      \node[vertex, label=left:$g^{-1}w$] (w) at (0, 0.5) {};
      \node[vertex, label=right:$v$] (gv) at (5, -1) {};
      \coordinate (x1) at (1, -1.5);
      \coordinate (x2) at (1.8, -0.6);
      \coordinate (x3) at (2.2, 0.8);
      \coordinate (y1) at (5.1, 0.8);
      \coordinate (y2) at (3.5, 0.8);
      \coordinate (y3) at (3, -1.4);
      \draw[edge] (v) -- (w) -- (gv);
      \draw[edge, blue] (v) -- (x1) -- (x2) -- (x3);
      \draw[edge, blue] (gv) -- (y1) -- (y2) -- (y3);
      
      \tkzLabelLine[blue, left, pos=0.8](x2, x3){$g^{-1}Q$}
      \tkzLabelLine[blue, left, pos=0.8](y2, y3){$g^{-1}Q'$}
      
      \node at (0.95, -0.4) {$g^{-1}D$};
      \node at (4.25, 0) {$g^{-1}D'$};
    \end{tikzpicture}
    \caption{$w \notin D'^\circ$}
  \end{subfigure}
  \caption{The minimal bounding paths $Q$ and $Q'$ intersecting $[v, w]$ and $[w, gv]$}
  \label{fig:membership}
\end{figure}
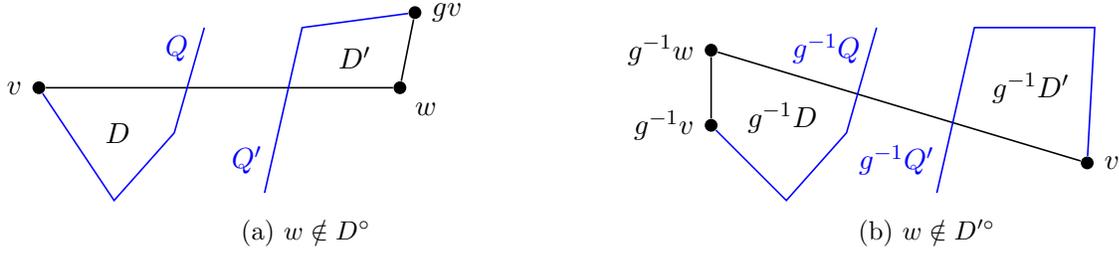

Suppose $X$ is reduced. Let $S$ be the set of short words of $X$.
For $g \in G$, let $\cH(g, v)$ be the half-space $\Set{w \in \Hyp \given d(w, v) < d(w, gv)}$.
The \emph{Dirichlet domain} \cite[\textsection 9.4]{beardon} for $G$ with centre $v$ is
\[
  D = D(G, v) = \Set{w \in \Hyp \given d(w, v) < d(w, gv) \text{ for all } g \in G \setminus \{1\}}.
\]
If $X$ is reduced then, by \Cref{prop:short-words-closer} (1),

\begin{align*}
  D &= \Set{w \in \Hyp \given d(w, v) < d(w, gv) \text{ for all } g \in S} \\
  &= \bigcap_{g \in S} \cH(g, v).
\end{align*}

We can therefore compute the sides of $D$.

\begin{figure}[t]
  \includegraphics[width=6cm]{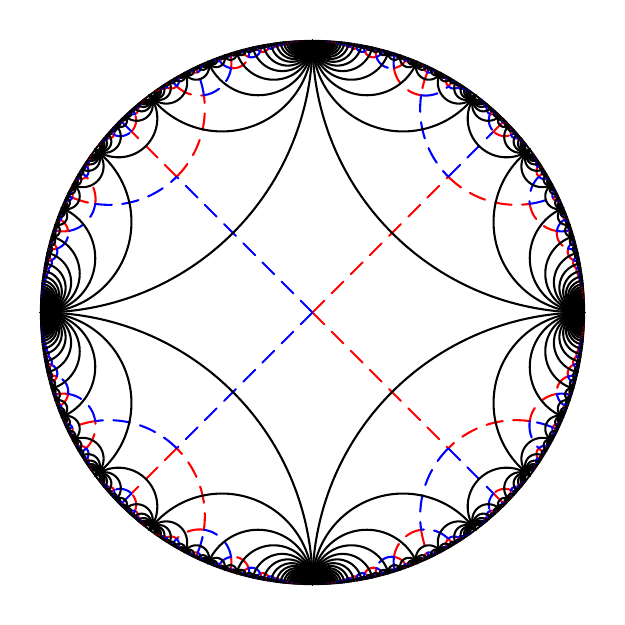}
  \caption{The Dirichlet domain centred at $v$ of the group shown in \Cref{fig:cayley-graph}}
  \label{fig:fundamental-domain}
\end{figure}

\begin{theorem}
  Let $X$ and $Y$ be finite subsets of $\PSL_2(\R)$. If $G = \gen{X}$ and $H = \gen{Y}$ are discrete and torsion-free, then there exists an algorithm that decides whether $G = H$.
\end{theorem}

\begin{proof}
  By the Poincaré Polygon Theorem \cite[\textsection 7.4]{stillwell}, $G = H$ if and only if $D(G, v) = D(H, v)$, which we can compute by the above reasoning.
\end{proof}

Given a fundamental domain $D$ with suitable properties, Eick, Kirschmer, and Leedham-Green \cite[Theorem 2.1]{eick} provide a solution to the constructive membership problem. They require that the complements of the half-spaces $H_g$ bordering $D$ are disjoint, which does not always hold in our case. However \Cref{prop:short-words-closer} provides an alternative set of conditions that is strong enough to obtain a similar algorithm.

\begin{theorem}\label{thm:alg-fund-domain}
  Suppose $X$ is reduced. Given $w \in \Hyp$, there is an algorithm that outputs a word $g$ in $X$ such that $gw$ is in the closure $\overline{D}$ of $D$. 
\end{theorem}
\Cref{alg:fund-domain} establishes this theorem. It works as follows: assign $g = 1$. Replace $g$ with $hg$, where $h$ is a short word in $X$ such that $d(v, hgw) < d(v, gw)$. Repeat this step until no such $h$ exists.

\begin{algorithm}
  \KwData{Reduced generating set $X$ of $\PSL_2(\R)$, $w \in \Hyp$}
  \KwOut{$g \in \gen{X}$ such that $gw \in \overline{D(G, v)}$}
  $g \gets 1$\;
  \While{there exists a short word $h$ of $X$ such that $d(v, hgw) < d(v, gw)$}{
    $g \gets hg$\;
  }
  \Return{$g$}\;
  \caption{Find word mapping vertex to Dirichlet domain}\label{alg:fund-domain}
\end{algorithm}

\begin{proof}[Proof of correctness of \Cref{alg:fund-domain}]
  Suppose $w \notin \overline{D}$. Then $d(w, gv) < d(w, v)$ for some $g \in G \setminus \{1\}$. By \Cref{prop:short-words-closer}, there is a short word $h$ of $X$ such that $d(hw, v) < d(w, v)$. By repeating these steps, we obtain a sequence $(1 = g_1, g_2, \dots$) of words in $X$ such that $d(g_{i+1} w, v) < d(g_{i} w, v)$ for all $i$.
  Since $G$ is discrete, this sequence must be finite, so there is some $n \in \N$ such that $g_n w \in \overline{D}$.
\end{proof}

\begin{remark}
If $gw \in \overline{D}$, then $Gw \cap \overline{D} = \Set{hgw \given h \in U}$ where $U = \Set{h \in S \given hgw \in H_h}$. We see that $gw$ is the unique representative of $w$ in $\overline{D}$ if and only if $gw \in D$. We can modify \Cref{alg:fund-domain} to return $hgw$, where $h \in U$ minimises $hgw$ with respect to this ordering. Hence all points in the same orbit of $G$ are mapped to the same representative in $\overline{D}$.
\end{remark}

\begin{theorem}
  Let $P$ be $\SL_2(\R)$ or $\PSL_2(\R)$. Let $X$ be a finite subset of $P$ generating a discrete torsion-free group $G$, and let $g \in P$. There exists an algorithm that decides whether $g$ can be written as a word in $X$, and if so returns such a word.
\end{theorem}

\begin{proof}
  We assume that $X$ is reduced, since \Cref{alg:free} writes the elements of a reduced generating set of $G$ as words in $X$. Using \Cref{alg:fund-domain},
  we find a word $h$ of $X$ such that $h^{-1}gv \in D$. If $g \in G$, then $h^{-1}gv = v \in D$, so $h$ is unique. Hence $g \in G$ if and only if $g = h^{-1}$.
\end{proof}

\section{Discrete subgroups of $\SL_2(\R)$ and $\PSL_2(\R)$}

Recall Selberg's Lemma \cite{selberg}.
\begin{theorem}
  Every finitely generated matrix group over a field of characteristic $0$ has a finite index torsion-free subgroup.
\end{theorem}
Given a finitely generated matrix group $G$ and prime $p$, let $\psi_p$ be the reduction map of $G$ modulo $p$. The proof of Selberg's Lemma finds a prime $p$ such that $\ker \psi_p$ is torsion-free. An algorithm to compute such a $p$ is in \cite[Proposition 3.4, Example 3.6]{detinko}.

We combine this with Schreier's Lemma \cite[Theorem 1.12]{cameron}.
\begin{theorem}\label{thm:schreier}
  Let $G$ be a finitely generated group with generating set $X$, and let $H$ be a finite-index subgroup of $G$ with $S$ a set of left coset representatives of $H$ is $G$. For all $g \in G$, define $\overline{g}$ to be the coset representative of $g$ in $S$. Then $H$ is generated by $\Set{\overline{(xs)^{-1}}xs \given s \in S,\; x \in X}$.
\end{theorem}

From these lemmas we conclude the following:

\recognition*
\begin{proof}
  The lemmas of Selberg and Schreier give a generating set for a finite index torsion-free subgroup $H$ of $G = \gen{X}$. We then use Algorithm \ref{alg:free} to decide whether or not $H$ is discrete. Since $H$ is discrete if and only if $G$ is discrete, we thus decide the discreteness of $G$.
\end{proof}

\membership*
\begin{proof}
  Let $g \in \PSL_2(\R)$. As above, we find a finite index torsion-free subgroup $H$ of $G$ and a set $S$ of left coset representatives of $H$ in $G$. Observe that $g \in G$ if and only if some element of $Sg$ is in $H$. Testing this immediately solves the constructive membership problem.
\end{proof}

Finally, we identify the Dirichlet domain $D(G, v)$. Riley \cite{riley} and Voight \cite{voight} give algorithms to compute the Dirichlet domain of a Fuchsian group. Here we present a simplified algorithm that, in our case, takes advantage of the information known about $G$.

Recall that $\cH(g, v)$ is the half-space $\Set{w \in \Hyp \given d(v, w) < d(w, g)}$. Observe that for all $g, h \in G$,
\begin{align*}
\cH(gh, v) &= \Set{w \in \Hyp \given d(v, w) < d(w, ghv)} \\
&= \Set{w \in \Hyp \given d(g^{-1}v, g^{-1}w) < d(g^{-1}w, hv)} \\
&= \Set{gw \in \Hyp \given d(g^{-1}v, w) < d(w, hv)} \\
&= g\cH(h, g^{-1}v).
\end{align*}

Let $H$ be a finite index torsion-free subgroup of $G$, and let $S$ be a set of left coset representatives of $H$. Now
\[
  D(G, v) = \bigcap_{g \in G}\cH(g, v)
  = \bigcap_{s \in S,\; h \in H}\cH(sh, v)
  = \bigcap_{s \in S}\left(s\bigcap_{h \in H}\cH(h, s^{-1}v)\right)
  = \bigcap_{s \in S}sD(H, s^{-1}v).
\]

We have thus expressed $D(G, v)$ in terms of finitely many Dirichlet domains for $H$, each of which can be computed as shown in \Cref{sec:membership}. We now compute a minimal set of generators and relations using \cite[Algorithm 5.7]{voight}.

\section{Complexity analysis}\label{sec:complexity}
To compare with Riley's algorithm in \cite{riley}, we briefly discuss the complexity of Algorithm \ref{alg:free} and the algorithm given in the proof of theorem \ref{recognition}. We consider all field operations and comparisons to be constant time computations.

\begin{proposition}
  The number of steps to complete Algorithm \ref{alg:free} with generating set $X$ is $O(\abs{X}^4 l)$, where $l$ is the length of the maximum word in $X$ that must be considered by the algorithm.
\end{proposition}

\begin{proof}
  In each iteration, we must partition the generating set $X_i$ into principal words, and then for each subword we iterate over its terms. Hence the complexity is bounded by the sum of the cube of the size of each partition, which is bounded by $(2\abs{X})^3 r$, where $r$ is the number of replacements performed before the algorithm terminates. We now give an upper bound for $r$ in terms of $l$. We see that $l$ is at least the average length of all words in the final generating set, which is at least $r/\abs{X}$ since each replacement increases the length of some word. Hence $r \leq \abs{X}l$. Therefore the algorithm has complexity $O(\abs{X}^4 l)$.
\end{proof}
Proposition \ref{prop:indiscrete-reduce-bound} gives an upper bound for $l$ in terms of the maximum displacement of an element of $X$. While this bound is at least exponential, we could not find a family of examples that meets it. In practice, $l$ seems to be much smaller than this bound would suggest.

The complexity of the algorithm exhibiting Theorem \ref{recognition} depends on the number of generators of the torsion-free subgroup $H$ of $G$ that is found. Since $G/H < \PSL_2(F_q)$ for some $q$, theorem \ref{thm:schreier} implies the upper bound $\abs{X}[G : H] \leq \abs{X}q^3$. Hence the number of steps to complete this algorithm is $O(\abs{X}^4q^{12}l)$.

\section{Implementation}\label{sec:magma}
We have implemented Algorithms \ref{alg:free} and \ref{alg:fund-domain}, and the algorithms described in the proofs of Theorems \ref{recognition} and \ref{membership}, for real algebraic number fields in the SL2RDiscrete package \cite{markowitz-magma} developed in \Magma. In this implementation we represent $\Hyp$ by the complex upper half plane, and fix $v = i$. We represent the generating set as a sequence rather than a set to simplify the code; this does not affect the algorithm.

\Cref{tab:recognition} reports the performance of \Cref{alg:free} for randomly chosen generating sets of $\SL_2(\Q(\sqrt{3}))$. For each trial, we chose a different pair of parameters $(n, d)$ that change the selected generating sets: $n$ is the size of the generating set, and each element is the product of $d$ random upper or lower triangular matrices. For each $n$, we empirically chose $d$ so that approximately half of the sets chosen generate a discrete torsion-free group; this seems to maximise the average runtime. For each pair of parameters, we give the average runtime over 1000 randomly chosen generating sets. The times for each trial vary significantly, so we also record the maximum runtime.

\Cref{tab:domain} reports the performance of \Cref{alg:fund-domain}. For each of 100 finite subsets of $\SL_2(\Q(\sqrt{3}))$ randomly chosen with parameters $(5, 15)$ that generate a discrete free group, we chose $10$ random points of $\Hyp$ and averaged the runtime over all 1000 pairs of points and generating sets. For each group $G$, we chose points of the form $gv$, where $v = i$ in the upper half-plane model of $\Hyp$ and $g$ is a random word with $m$ terms in the reduced generating set of $G$ found via \Cref{alg:free}. \Cref{tab:domain} reports the average runtime per trial for different values of $m$.

Our implementations of the algorithms exhibiting Theorems \ref{recognition} and \ref{membership} often take much longer to run even for small examples, due to the large index of the torsion-free subgroup that is found. For example, let $G \cong \Delta(2, 6, 6)$ be generated by
\[
A = \begin{bmatrix}
  2-t & t-1 \\
  2-2 t & t-2
\end{bmatrix}, \quad
B = \begin{bmatrix}
  t-1 & 1 \\
  t-2 & 1
\end{bmatrix}, \quad
C = \begin{bmatrix}
  t & 1 \\
  -1 & 0
\end{bmatrix},
\]
where $t = \sqrt{3}$ (see \cite{mcmullen}).
Our package finds a torsion-free subgroup $H$ of index 240 in $G$ with 166 generators, and returns a reduced generating set for $H$ with 22 generators after 160 minutes.
It also confirms that the matrix
\[
  \begin{bmatrix}
   -12483933055t - 21622285764 & 12710965t + 21803060 \\
   23447531t + 40614766 & -23470t - 41654
  \end{bmatrix}
\]
is not in $G$ in 8.1 minutes.

The tests were run in \Magma\ V2.28-4 on an Intel Xeon E5-2690 v4 2.60GHz CPU. The code used to generate these tests is in the file ``times.m'' in \cite{markowitz-magma}.

\begin{table}[h]
  \begin{tabular}{|r|r r r r|}
    \hline
    $(n, d)$ & $(2, 7)$ & $(5, 15)$ & $(10, 20)$ & $(20, 27)$ \\
    \hline
    Average time & 0.003 & 0.055 & 0.78 & 21 \\
    Maximum time & 0.04 & 0.91 & 23 & 440 \\
    \hline
  \end{tabular}
  \caption{Runtime for \Cref{alg:free} (seconds)}
  \label{tab:recognition}
\end{table}

\begin{table}[h]
  \begin{tabular}{|r|r r r r|}
    \hline
    $m$ & $5$ & $10$ & $20$ & $40$ \\
    \hline
    Average time & 0.15 & 0.39 & 1.2 & 4.7 \\
    \hline
  \end{tabular}
  \caption{Runtime for \Cref{alg:fund-domain} (seconds)}
  \label{tab:domain}
\end{table}

\section{Proof of \Cref{prop:indiscrete-reduce-bound}}\label{ap:bound}
We apply standard results of neutral geometry (that is, geometry without the parallel axiom), as found for example in \cite[\textsection 2]{hartshorne}.
Given points $A$, $B$, and $C$, define $A \btw B \btw C$ to mean that $B$ is between $A$ and $C$. We use $\cong$ to indicate congruence. Arithmetic operations on line segments are applied to their lengths.

\begin{lemma}\label{lem:quad-bisector}
  Let $ABCD$ be a convex quadrilateral. Let $E$ be the intersection of the diagonals $AC$ and $BD$. Let $l$ be the perpendicular bisector of $CD$. Suppose the rays $\overrightarrow{DA}$ and $\overrightarrow{CB}$ intersect $l$ at points $F$ and $G$ respectively. Then $\max\{AF, BG\} \geq AB/2$. Furthermore, if $AC \cong BD$ and $AE + DE \leq AC$, then $AF \geq AB/2$.
\end{lemma}

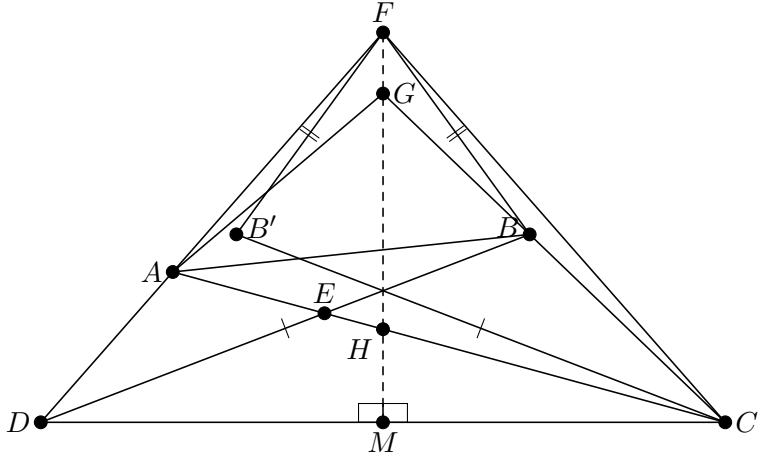
\begin{figure}[h]
  \centering
  \begin{tikzpicture}[xscale=1.3]
    \tkzDefPoints{0.35/2/A, 4/2.5/B, 6/0/C, -1/0/D}
    
    \tkzDefMidPoint(C,D) \tkzGetPoint{M}
    \tkzInterLL(A,C)(B,D) \tkzGetPoint{E}
    
    \tkzDefLine[mediator, normed](D,C)
    \tkzGetPoint{M'}
    \tkzInterLL(D,A)(M,M') \tkzGetPoint{F}
    \tkzInterLL(C,B)(M,M') \tkzGetPoint{G}
    
    \tkzInterLL(C,A)(M,M') \tkzGetPoint{H}
    
    \tkzDefPointBy[reflection=over M--M'](B) \tkzGetPoint{B'}
    
    \tkzDrawSegments[edge](A,B C,D A,C B,D D,F C,G A,G F,C B',F B',C B,F)
    \tkzDrawSegment[edge, dashed](M, F)
    
    \tkzDrawPoints[vertex](A,B,C,D,E,M,F,G,H,B')
    \tkzLabelPoint[left](A){$A$}
    \tkzLabelPoint[left, shift=({0, 0.1})](B){$B$}
    \tkzLabelPoint[right](C){$C$}
    \tkzLabelPoint[left](D){$D$}
    \tkzLabelPoint[above](E){$E$}
    \tkzLabelPoint[below](M){$M$}
    \tkzLabelPoint[above](F){$F$}
    \tkzLabelPoint[right](G){$G$}
    \tkzLabelPoint[below left](H){$H$}
    \tkzLabelPoint[right, shift=({0, 0.1})](B'){$B'$}
    
    \tkzMarkRightAngles(D,M,F C,M,G)
    \tkzMarkSegments(B',C B,D)
    \tkzMarkSegments[mark=||](B',F B,F)
  \end{tikzpicture}
  \caption{Diagram for the proof of \Cref{lem:quad-bisector}}
  \label{fig:quad-bisector}
\end{figure}

\begin{proof}
  Assume without loss of generality that $CG \leq DF$. Let $M$ be the midpoint of $CD$. Note that $A$ is not in the interior of the triangle $\Delta CGM$ (since otherwise $AD$ intersects $GM$, so $DF < CG$), while $AC$ passes through the interior of $\Delta CGM$ so, by the Crossbar Theorem \cite[Proposition 7.3]{hartshorne}, $AC$ intersects $GM$ at some point $H$. Since $M \btw H \btw G$ and $M \btw G \btw F$, it follows that $H \btw G \btw F$. Now
  \[
  \angle AGF \geq \angle AHF \cong \angle MHC \geq \angle MFC \cong \angle GFA.
  \]
  Therefore $AF \geq AG$, so $AF + BG \geq AG + BG \geq AB$.
  
  Now suppose that $AC \cong BD$ and $AE + DE \leq AC$. Let $B'$ be the reflection of $B$ across $l$. Note that $B$ lies in $\angle ADC = \angle FDC$ so, by reflecting across $l$, we find that $\angle B'CD \leq \angle FCD$. Additionally, $DE \leq AC - AE = CE$, so $\angle ACD = \angle ECD \leq EDC \cong B'CD$.
  Hence $A$ lies in $\angle B'CD$, so $B'$ lies in $\angle FCA$. Since $B'C \cong AC$, by the Hinge Theorem \cite[p. 102]{hartshorne} $AF \geq B'F \cong BF$. Since $AF + BF \geq AB$, we conclude that $AF \geq AB/2$.
\end{proof}

\begin{lemma}\label{lem:quad-worst-case}
  Suppose $ABCD$ be a convex quadrilateral. Let $E$ be the intersection of the diagonals $AC$ and $BD$. Let $A'$ and $B'$ be points on the rays $\overrightarrow{DA}$ and $\overrightarrow{CB}$ respectively such that $D \btw A \btw A'$, $C \btw B \btw B'$, and $AA', BB' \leq AB/2$. Either $DF \leq CF$ or $CG \leq DG$. Furthermore, if $AC \cong BD$ and $AE + ED \leq AC$, then $A'D \leq A'C$.
\end{lemma}

\begin{figure}[h]
  \centering
  \begin{tikzpicture}
    \tkzDefPoints{0.35/2/A, 3.2/2.5/B, 6/0/C, -1/0/D}
    
    \tkzDefMidPoint(C,D) \tkzGetPoint{M}
    \tkzInterLL(A,C)(B,D) \tkzGetPoint{E}
    
    \tkzDefLine[mediator, normed](D,C)
    \tkzGetPoint{M'}
    \tkzDefPointOnLine[pos=7](M',M) \tkzGetPoint{N}
    \tkzDefPointOnLine[pos=1.4](D,A) \tkzGetPoint{A'}
    \tkzDefPointOnLine[pos=1.4](C,B) \tkzGetPoint{B'}
    
    \tkzInterLL(D,A)(M,M')
    \tkzGetPoint{F}
    
    \tkzInterLL(C,B)(M,M')
    \tkzGetPoint{G}
    
    \tkzDrawSegments[edge](D,F C,B' C,D A,B A,C B,D)
    \tkzDrawSegment[edge, dashed, ->](M, N)
    
    \tkzDrawPoints[vertex](A,B,C,D,E,A',B',F,G)
    \tkzLabelPoint[left](A){$A$}
    \tkzLabelPoint[right](B){$B$}
    \tkzLabelPoint[right](C){$C$}
    \tkzLabelPoint[left](D){$D$}
    \tkzLabelPoint[above](E){$E$}
    \tkzLabelPoint[above](A'){$A'$}
    \tkzLabelPoint[above](B'){$B'$}
    \tkzLabelPoint[left](F){$F$}
    \tkzLabelPoint[right](G){$G$}
    
    \tkzMarkRightAngles(D,M,N C,M,N)
  \end{tikzpicture}
  \caption{Diagram for the proof of \Cref{lem:quad-worst-case}}
  \label{fig:quad-worst-case}
\end{figure}
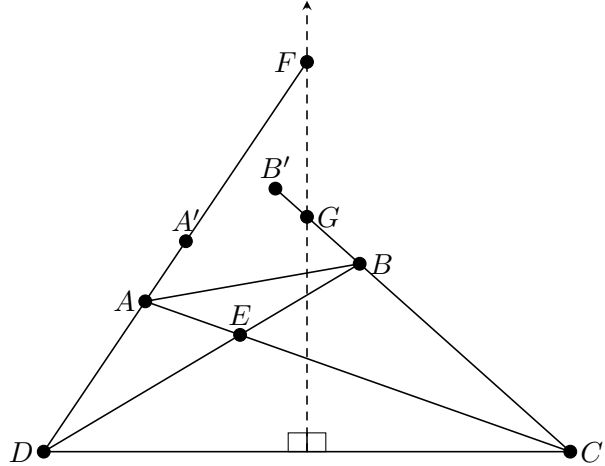

\begin{proof}
  Let $l$ be the perpendicular bisector of $CD$.
  If $\overrightarrow{DA}$ and $\overrightarrow{CB}$ intersect $l$ at points $F$ and $G$ respectively, then \Cref{lem:quad-bisector} implies that $\max\{AF, BG\} \geq AB/2$. Hence we may assume without loss of generality (by relabelling the vertices of $ABCD$) that either $\overrightarrow{DA}$ does not intersect $l$, or it intersects $l$ at some point $F$ such that $D \btw A \btw F$ and $AF \geq AB/2$, in which case $A \btw A' \btw F$. In either case, $A'$ lies to the same side of $l$ as $D$, so $A'D \leq A'C$. If $AC \cong BD$ and $AE + DE \leq AC$, then by \Cref{lem:quad-bisector} the above argument applies without any relabelling.
\end{proof}

\begin{lemma}\label{lem:quad-reduce-bound}
  Let $\e > 0$. Let $BCD$ be a triangle. Let $A$ be a point on $BD$. If $CD \geq \e$, $BD \leq BC$ and $AB = \e/2$, then
  \[
  \cosh(AC) - \cosh(AD) > 2\sinh^3\left(\frac{\e}{2}\right)\csch\left(AC + \frac{\e}{2}\right).
  \]
\end{lemma}

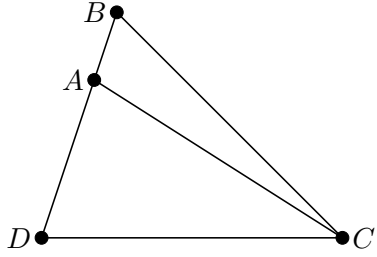
\begin{figure}[h]
  \centering
  \begin{tikzpicture}
    \tkzDefPoints{1/3/B, 4/0/C, 0/0/D}
    
    \tkzDefPointOnLine[pos=0.3](B,D)\tkzGetPoint{A}
    
    \tkzDrawSegments[edge](B,C C,D D,B A,C)
    \tkzDrawPoints[vertex](A,B,C,D)
    \tkzLabelPoint[left](A){$A$}
    \tkzLabelPoint[left](B){$B$}
    \tkzLabelPoint[right](C){$C$}
    \tkzLabelPoint[left](D){$D$}
  \end{tikzpicture}
  \caption{Diagram for the proof of \Cref{lem:quad-reduce-bound}}
  \label{fig:quad-reduce-bound}
\end{figure}

\begin{proof}
  Let $\alpha = \angle ADC = \angle DBC$. By the hyperbolic cosine rule \cite[\textsection 7.12]{beardon},
  \begin{gather}
    \cosh(AC) = \cosh(AD)\cosh(CD) - \sinh(AD)\sinh(CD)\cos(\alpha), \\
    \cosh(BC) = \cosh\left(AD + \frac{\e}{2}\right) \cosh(CD) - \sinh\left(AD + \frac{\e}{2}\right)\sinh(CD)\cos(\alpha).
  \end{gather}
  
  We solve (2) for $\cos(\alpha)$ and substitute into (1), obtaining
  {\allowdisplaybreaks
    \begin{align*}
      \cosh(AC) &= \cosh(AD)\cosh(CD) - \sinh(AD)\frac{\cosh\left(AD + \frac{\e}{2}\right) - \cosh(BC)}{\sinh\left(AD + \frac{\e}{2}\right)} \\
      &= \frac{\cosh(CD)\left(\cosh(AD)\sinh\left(AD + \frac{\e}{2}\right) - \sin(AD)\cosh\left(AD + \frac{\e}{2}\right)\right) + \sinh(AD)\cosh(BC)}{\sinh\left(AD + \frac{\e}{2}\right)} \\
      &= \frac{\cosh(CD)\sinh\left(\frac{\e}{2}\right) + \sinh(AD)\cosh(BC)}{\sinh\left(AD + \frac{\e}{2}\right)} \\
      &\geq \frac{\cosh(\e)\sinh\left(\frac{\e}{2}\right) + \sinh(AD)\cosh\left(AD + \frac{\e}{2}\right)}{\sinh\left(AD + \frac{\e}{2}\right)} \\
      &= \frac{\cosh(\e)\sinh\left(\frac{\e}{2}\right) + \sinh(AD)\cosh\left(AD + \frac{\e}{2}\right) - \cosh(AD)\sinh\left(AD + \frac{\e}{2}\right)}{\sinh\left(AD + \frac{\e}{2}\right)} + \cosh(AD) \\
      &= \frac{\cosh(\e)\sinh\left(\frac{\e}{2}\right) - \sinh\left(\frac{\e}{2}\right)}{\sinh\left(AD + \frac{\e}{2}\right)} + \cosh(AD) \\
      &= \sinh\left(\frac{\e}{2}\right)\left(\cosh(\e) - 1\right)\csch\left(AD + \frac{\e}{2}\right) + \cosh(AD) \\
      &= 2\sinh^3\left(\frac{\e}{2}\right)\csch\left(AD + \frac{\e}{2}\right) + \cosh(AD).
    \end{align*}
  }
  This proves that $AC > AD$, so
  \begin{align*}
    \cosh(AC) - \cosh(AD) &\geq 2\sinh^3\left(\frac{\e}{2}\right)\csch\left(AD + \frac{\e}{2}\right) \\
    &> 2\sinh^3\left(\frac{\e}{2}\right)\csch\left(AC + \frac{\e}{2}\right). \qedhere
  \end{align*}
\end{proof}

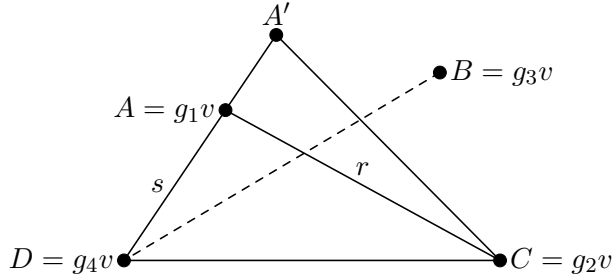
\begin{figure}[h]
  \centering
  \begin{tikzpicture}
    \tkzDefPoints{0.35/2/A, 3.2/2.5/B, 4/0/C, -1/0/D}
    
    \tkzDefMidPoint(C,D)
    
    \tkzDefPointOnLine[pos=1.5](D,A) \tkzGetPoint{A'}
    
    \tkzDrawSegments[edge](D,A' C,D A,C A',C)
    \tkzDrawSegment[edge, dashed](B,D)
    
    \tkzDrawPoints[vertex](A,B,C,D,A')
    \tkzLabelPoint[left](A){$A = g_1 v$}
    \tkzLabelPoint[right](B){$B = g_3 v$}
    \tkzLabelPoint[right](C){$C = g_2 v$}
    \tkzLabelPoint[left](D){$D = g_4 v$}
    \tkzLabelPoint[above](A'){$A'$}
    
    \tkzLabelLine[left](A,D){$s$}
    \tkzLabelLine[above](A,C){$r$}
  \end{tikzpicture}
  \caption{Diagram for the proof of \Cref{prop:indiscrete-reduce-bound}}
  \label{fig:indiscrete-reduce-bound}
\end{figure}

\begin{proof}[Proof of \Cref{prop:indiscrete-reduce-bound}]
  Suppose (1), (2) and (3) do not hold.
  There exists a minimal path $P$ containing a short self-intersecting segment $\gamma$ of $T$ with intersecting edges $e_1$ and $e_2$. Let $x$ and $y$ be the types of $e_1$ and $e_2$ respectively. If $e_1$ and $e_2$ are incident to a common vertex, then their endpoints must be collinear, so $\abs{xy} < \max\{\abs{x}, \abs{y}\}$. Hence $y \neq x^{-1}$ and $y \neq x$ (since this would imply that $x$ is elliptic), so $xy$ is a good replacement for either $x$ or $y$. Assume without loss of generality that $xy$ is a good replacement for $x$. Then $\abs{x} - \abs{xy} = \abs{y} \geq \e$.
  
  Henceforth we assume that $e_1$ and $e_2$ are not incident to a common vertex. Let $\{g_1, g_2\}$ and $\{g_3, g_4\}$ be the endpoints of $e_1$ and $e_2$ respectively. Define $A = g_1 v, C = g_2 v, B = g_3 v, D = g_4 v$. We see that $ABCD$ is a convex quadrilateral, since the diagonals $AC$ and $BD$ intersect.
  
  The following cases are the same as those in the proof of \Cref{lem:half-loop-reduces}; we similarly use Figure \ref{fig:replacements} for reference.
  First suppose that either $x \neq y$ or $x^{-1}$ is not the type of an edge in $P$. One of the cases in Figure \ref{fig:replacements} must hold, so we may assume (up to relabelling) that exactly one term in $g_1^{-1} g_4$ is either $x$ or $x^{-1}$, and exactly one term in $g_2^{-1}g_3$ is either $y$ or $y^{-1}$. Let $A'$ be the point on $\overrightarrow{DA}$ such that $D \btw A \btw A'$ and $AA' = \e/2$. By \Cref{lem:quad-worst-case}, we may assume that $A'D \leq A'C$. Let $r$ and $s$ be the lengths of $AD$ and $AC$ respectively. We see that $r = \abs{z}$ and $s = \abs{g}$ for some good replacement $g \in \{g_1^{-1} g_4, g_2^{-1}g_3\}$ for $z \in \{x, y\}$. \Cref{lem:quad-reduce-bound} thus implies that $\cosh(r) - \cosh(s) > \eta(r, \e) > 0$. Solving for $s$, we find
  \[
  s < \cosh^{-1}(\cosh(r) - \eta(r, \e)) < r,
  \]
  from which the result follows.
  
  Now suppose that $x = y$ and that $x^{-1}$ is the type of an edge in $P$. In this case, we choose our labelling so that $g_1$ and $g_4$ are the endpoints of $P$, in which case exactly one term in $g_1^{-1}g_4$ is either $x$ or $x^{-1}$. Let $F \in [A, B] \cap [C, D]$ be the image under $\phi$ of the endpoints of $\gamma$. Then $AF + FD \leq AB$. Defining $A'$ as above, by \Cref{lem:quad-worst-case} $A'D \leq A'C$. The rest follows as above.
\end{proof}
In fact $\e \geq r - \cosh^{-1}(\cosh(r) - \eta(r,\e))$ for all $r \geq \e > 0$, so the definition of $\delta$ in \Cref{prop:indiscrete-reduce-bound} can be simplified; however we omit the proof.

\section{Acknowledgements}
Several figures in this paper were created using the HyperbolicPlane Julia module \cite{scheinerman}.
Thanks to Matthew Conder, Eamonn O'Brien and Jeroen Schillewaert for their feedback on drafts of this paper. Also thanks to the University of Auckland Doctoral Scholarship and the Marsden Fund (grant number 22-UOA-194) for funding this research.

\bibliographystyle{plain}
\bibliography{references.bib}

\end{document}